\documentclass[11pt]{amsart}
\usepackage{latexsym}
\usepackage{fullpage}
\usepackage{amssymb}
\usepackage{amscd}
\usepackage{float}
\usepackage{graphicx}
\usepackage{amsfonts}
\usepackage{pb-diagram}
 \usepackage{amsmath,amscd}

\newtheorem{thm}{Theorem}
\newtheorem{cor}{Corollary}
\newtheorem{lemma}{Lemma}
\newtheorem{prop}{Proposition}
\newtheorem{defn}{Definition}
\newtheorem{remark}{Remark}

\newtheorem{ex}{Example}
\newtheorem{nt}{Notation}

\begin{document}

\title[The HOMFLYPT skein module of $S^1 \times S^2$ via braids]
  {The HOMFLYPT skein module of $S^1 \times S^2$ via braids}

\author{Ioannis Diamantis}
\address{Department of Data Analytics and Digitalisation,
Maastricht University, School of Business and Economics,
P.O.Box 616, 6200 MD, Maastricht,
The Netherlands.
}
\email{i.diamantis@maastrichtuniversity.nl}



\setcounter{section}{-1}

\date{}

\begin{abstract}
In this paper we compute the HOMFLYPT skein module of $S^1 \times S^2\, \cong
\, L(0, 1)$, denoted $\mathcal{S}(S^1 \times S^2)$, using braid-theoretic techniques. We extend the Lambropoulou invariant, $X$, for links in the solid torus ST to links in $S^1 \times S^2$, by solving an infinite system of equations of the form $X_{\widehat{a}} = X_{\widehat{bbm(a)}}$, where $bbm(a)$ denotes all possible band moves applied to $a$, for all $a$ in a basis of $\mathcal{S}(ST)$. We show that the free part of $\mathcal{S}(S^1 \times S^2)$ is generated by the empty link, while all other elements are torsion. 
\end{abstract}

\maketitle

\section{Introduction}

Skein modules were independently introduced by Przytycki \cite{P} and Turaev \cite{Tu} as generalizations of knot polynomials in $S^3$ to knot polynomials in arbitrary 3-manifolds. The essence is that skein modules are quotients of free modules over ambient isotopy classes of links in 3-manifolds by properly chosen local (skein) relations. More precisely:

\begin{defn}\rm
Let $M$ be an oriented $3$-manifold, $R=\mathbb{C}[u^{\pm1},z^{\pm1}]$, $\mathcal{L}$ the set of all oriented links in $M$ up to ambient isotopy in $M$ and let $S$ be the submodule of $R\mathcal{L}$ generated by the skein expressions $u^{-1}L_{+}-uL_{-}-zL_{0}$, where
$L_{+}$, $L_{-}$ and $L_{0}$ comprise a Conway triple represented schematically by the illustrations in Figure~\ref{skein}.

\begin{figure}[h]
\begin{center}
\includegraphics[width=1.7in]{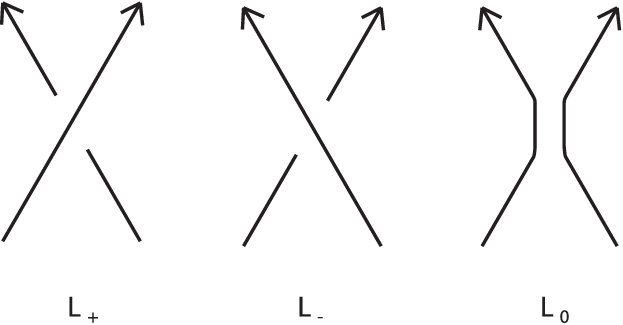}
\end{center}
\caption{The links $L_{+}, L_{-}, L_{0}$ locally.}
\label{skein}
\end{figure}

\noindent For convenience we allow the empty knot, $\emptyset$, and add the relation $u^{-1} \emptyset -u\emptyset =zT_{1}$, where $T_{1}$ denotes the trivial knot. Then the {\it HOMFLYPT skein module} of $M$ is defined to be:

\begin{equation*}
\mathcal{S} \left(M\right)=\mathcal{S} \left(M;{\mathbb C}\left[u^{\pm 1} ,z^{\pm 1} \right],u^{-1} L_{+} -uL_{-} -zL{}_{0} \right)={\raise0.7ex\hbox{$
R\mathcal{L} $}\!\mathord{\left/ {\vphantom {R\mathcal{L} S }} \right. \kern-\nulldelimiterspace}\!\lower0.7ex\hbox{$ S  $}}.
\end{equation*}
\end{defn}

\smallbreak

Braid-theoretic methods have become a central tool for studying skein modules of $3$-manifolds. The key idea is to represent links in a given manifold via \emph{mixed braids} in a simpler space (typically a solid torus) and then impose additional relations, such as \emph{braid band moves}, to recover the original manifold’s topology. This braid-based perspective has been successfully applied to compute various skein modules of 3-manifolds (see \cite{D, D1, DL2,DL3, DL4, DLP} for examples of HOMFLYPT skein modules and \cite{D2022CCM, D2023arXiv, D2023KBSM, D2} for Kauffman bracket skein modules of certain 3-manifolds via braids). In each case, translating knots and links into the language of braids (together with appropriate moves) provides an algebraic framework for managing the complexity of the skein relations.

One case of particular interest is the manifold $S^1 \times S^2$. Topologically, $S^1 \times S^2$ is the simplest example of a $3$-manifold whose skein module contains torsion, making its computation particularly subtle. The pioneering work of Gilmer and Zhong~\cite{GZ1} computed the HOMFLYPT skein module of $S^1 \times S^2$ diagrammatically, and, after localizing the coefficient ring to invert certain combinations of skein parameters, they found the module to be free of rank one, generated by the empty link. Later, Lavrov and Rutherford~\cite{LR3} approached the same problem through a Legendrian link perspective, constructing a well-defined HOMFLYPT-type invariant for links in $S^1 \times S^2$, and similarly observing that the skein module becomes free of rank one under localization. Both approaches rely on inverting specific elements of the ring, effectively removing torsion from the natural skein module. This fact motivated a fresh examination of $S^1 \times S^2$ using the braid-theoretic framework, which allows us to work over the full polynomial ring without localization. In particular, our computation reveals explicitly that the free part of the module is generated by the empty knot, while all other skein classes are torsion. Thus, our result complements the findings of Gilmer–Zhong and Lavrov–Rutherford, even though an explicit presentation of the torsion submodule via braids remains an open problem.

In this paper, we revisit the HOMFLYPT skein module of $S^1 \times S^2$ using a purely braid-theoretic and algebraic approach, extending techniques developed for other $3$-manifolds~\cite{DL3, DAMS}. We exploit the fact that $S^1 \times S^2$ can be described as $S^3$ with an unknot removed and reattached ($0$--framed surgery). Consequently, any link in $S^1 \times S^2$ can be represented as a \emph{mixed link} in the solid torus $ST$ (the complement of an unknot in $S^3$), together with an additional move corresponding to the $0$--surgery. More concretely, there is a well-known correspondence between links in $S^1 \times S^2$ and mixed links in the solid torus, $ST$, where isotopy in $S^1 \times S^2$ is equivalent to isotopy in $ST$ augmented by a finite sequence of {\it braid band moves} (bbm)~\cite{LR1,LR2,DL1}. The braid band move is an isotopy move specific to the surgery description of $S^1 \times S^2$: it consists of cutting one strand of the mixed braid and reconnecting it to another part of the braid (introducing a half-twist in the process) to simulate the effect of the $0$--framing. In essence, a band move on mixed braids corresponds to the second Kirby move in the topological description of $S^1 \times S^2$. Encoding these moves algebraically, allows us to translate the problem of computing $\mathcal{S}(S^1 \times S^2)$ into an algebraic one.

Our approach builds upon the HOMFLYPT-type invariant $X$ for links in the solid torus, originally constructed by Lambropoulou via a Markov trace on the generalized Hecke algebra of type B~\cite{La1}. This invariant $X$ assigns to every braid in $ST$ a polynomial analogous to the HOMFLYPT polynomial, and it is powerful enough to distinguish all basis elements of the skein module $\mathcal{S}(ST)$. To extend $X$ from the solid torus to the manifold $S^1 \times S^2$, we impose the condition that $X$ remains invariant under the braid band move. In other words, whenever a mixed braid $\alpha$ in $ST$ can undergo a band move to $\alpha'$ (we write $\text{bbm}(\alpha) = \alpha'$), we require $X$ to assign the same value to their closures in $S^1 \times S^2$. Equivalently, for every basis braid $\alpha$ in $\mathcal{S}(ST)$ and every possible braid band move on $\alpha$, we enforce the \emph{braid invariance condition}:

\begin{equation}\label{system}
    X\big(\widehat{\alpha}\big) \;=\; X\big(\widehat{\text{bbm}(\alpha)}\big),
\end{equation}

\noindent where $\widehat{\alpha}$ denotes the link in $S^1 \times S^2$ obtained by closing $\alpha$ in the mixed link diagram (and similarly for $\widehat{\text{bbm}(\alpha)}$). This condition translates into an infinite system of linear skein relations (essentially, equations in the polynomial variables $q$, $z$ and and the $s_i$'s) satisfied by the values of $X$ on the basis of $\mathcal{S}(ST)$. Solving this infinite system of equations is the key to computing the HOMFLYPT skein module of $S^1 \times S^2$. In practice, we work with a new basis for $\mathcal{S}(ST)$ (a specially chosen basis $\Lambda$ introduced in~\cite{DL2}) that is particularly convenient for handling braid band moves, so that the infinite relations can be organized and resolved systematically. By iteratively solving these relations, we explicitly express every element of $\mathcal{S}(S^1 \times S^2)$ in terms of a simplified generating set.

The outcome of our computation is a description of the free part of $\mathcal{S}(S^1 \times S^2)$, stated informally as follows:

\smallbreak

\emph{The free part of the HOMFLYPT skein module of $S^1 \times S^2$ is generated by the empty link (or the unknot), with all other elements appearing in a torsion submodule.}

\smallbreak

\noindent More precisely, over the ring $R = \mathbb{C}[q^{\pm1}, z^{\pm1}]$, one has a direct sum decomposition
\[
\mathcal{S}(S^1 \times S^2) \;\cong\; R\{\emptyset\} \;\oplus\; T,
\]
where $R\{\emptyset\}$ is the rank-$1$ free submodule generated by the empty link, and $T$ is a torsion submodule. This result resolves the apparent discrepancy with Gilmer-Zhong’s work: it confirms that one independent HOMFLYPT-type invariant exists in $S^1 \times S^2$ (generated by the unknot), while also explicitly identifying the torsion that was invisible when working over a localized coefficient ring.

\bigbreak

\noindent {\bf Outline of the method.} For the reader’s convenience, we summarize here the main steps of our computation:

\begin{itemize}
\item[$\bullet$] We start from elements in the standard by now basis of $\mathcal{S}({\rm ST})$, $\Lambda$, presented in \cite{HK}. Then, following the technique in \cite{DL2}, we express these elements into sums of elements in an augmented set, $\Lambda^{aug}$, using conjugation and stabilization moves.
\item[$\bullet$] We obtain an infinite system of equations by performing bbm's on the $1^{st}$ moving strand of the elements in $\Lambda^{aug}$, obtained in the previous step.
\item[$\bullet$] Using now the inverse of the change of basis matrix presented in \cite{DL2}, we express these elements into sums of elements in the basic set of ST, $\Lambda^{\prime}$.
\item[$\bullet$] Then, the following diagram is shown to commute:
\[
\begin{array}{ccccc}
s & \overset{tr}{\longleftarrow} & \left({\Lambda^{aug}}\right)^{\prime} \ni \tau^{\prime} & {\widehat{\cong}} & a\cdot \tau + \underset{i}{\Sigma} a_i\cdot \tau_i \\
\uparrow  &                              & \overset{bbm_1}{\downarrow}                &                   & \overset{bbm_i}{\downarrow} \\
|	&                              &  bbm_1(\tau^{\prime})                      & {\widehat{\cong}} &  a\cdot \tau_+\cdot \delta_i^{\pm 1} + \underset{i}{\Sigma} a_i\cdot \tau_{i_+}\cdot \delta_i^{\pm 1}\\
|	&                              &                  |                          &                   & {\widehat{\cong}} \\
|	&                              &                |                  &                   &  \underset{j}{\Sigma} b_j\cdot \tau_{j_+}\cdot \sigma_1^{\pm 1}  \\
	
\downarrow  &                              &             \downarrow                               &                   & {\widehat{\cong}} \\
	\underset{i}{\Sigma} c_i\cdot s_i & \overset{tr}{\longleftarrow} & \underset{i}{\Sigma} c_i\cdot \tau_i^{\prime} & {\widehat{\cong}} & \underset{u}{\Sigma} a_u \cdot \tau_u \\
\end{array}
\]

\item[$\bullet$] The sets $S$ and $S_k$ are now introduced (Definition~\ref{Ssets}), which consist of monomials in $s_i$'s, the unknowns of the system (\ref{system}).
\item[$\bullet$] An ordering relation is then defined on these sets and it is shown that the sets $S_k$ equipped with this ordering relation are totally ordered sets.
\item[$\bullet$] It is then shown that the infinite system of Eq.~(\ref{system}) splits into infinitely many self-contained subsystems of equations, each one corresponding to a set $S_k, k\in \mathbb{Z}$.
\item[$\bullet$] For an arbitrary subsystem obtained from elements in $\Lambda_k^{aug}$, it is shown that all unknowns in $S_k$ can be written as combinations of the minimum element in $S_k$ and thus, they are torsion elements in $\mathcal{S}(S^1 \times S^2)$.
\item[$\bullet$] The minimum element in $S_k$ is also shown to be a torsion element in $\mathcal{S}(S^1 \times S^2)$.
\item[$\bullet$] Finally, it is concluded that the free part of $\mathcal{S}(S^1 \times S^2)$ is generated by the empty link.
\end{itemize}

\bigbreak

Our computation of $\mathcal{S}(S^1 \times S^2)$ via the braid approach fits into a broader program of using braids to understand skein modules. Similar braid-theoretic techniques have been used recently to tackle the HOMFLYPT skein modules of the solid torus and of lens spaces~\cite{DL3,DL4, DGLM}, as well as the Kauffman bracket skein modules of various $3$-manifolds via mixed braids~\cite{D2022CCM, D2023arXiv, D2023KBSM}. These developments underscore the flexibility and strength of the braid approach: by recasting topological problems in algebraic terms, one can solve otherwise intractable skein equations and even uncover new algebraic phenomena.

The paper is organized as follows: In \S~\ref{basics} we recall the setting and essential braid-theoretic techniques, including isotopy moves and braid equivalence in $L(p, 1)$. In \S~\ref{newbas} we present an alternative basis $\Lambda$ for $\mathcal{S}(ST)$, introduced in \cite{DL2}, which allows for a natural formulation of braid band moves. In \S~\ref{HOMS} we describe the infinite system of equations obtained from braid band moves and we demonstrate that solving this system yields the HOMFLYPT skein module of $S^1 \times S^2$.

\section{Algebraic and topological setting}\label{basics}

\subsection{Mixed links and isotopy in $L(0,1)$}
 
We consider $\mathrm{ST}$ to be the complement of a solid torus in $S^3$. As explained in \cite{LR1, LR2, DL1}, an oriented link $L$ in $\mathrm{ST}$ can be represented by an oriented \textit{mixed link} in $S^{3}$, that is, a link in $S^{3}$ consisting of the unknotted fixed part $\widehat{I}$ representing the complementary solid torus in $S^3$, and the moving part $L$ that links with $\widehat{I}$. A \textit{mixed link diagram} is a diagram $\widehat{I} \cup \widetilde{L}$ of $\widehat{I} \cup L$ on the plane of $\widehat{I}$, where this plane is equipped with the top-to-bottom orientation of $I$ (for an illustration see Figure~\ref{mli}).

\begin{figure}[H]
\begin{center}\includegraphics[width=1.4in]{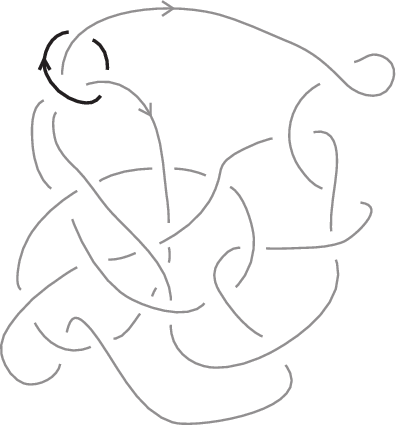}
\end{center}
\caption{A link in ST as a mixed link in $S^3$.}
\label{mli}
\end{figure}

It is well known that the lens space $L(0,1) \cong S^1 \times S^2$ can be obtained from $S^3$ by surgery on the unknot with surgery coefficient $0$. Surgery along the unknot can be realized by considering first the complementary solid torus and then attaching to it another solid torus according to a homeomorphism on their boundaries. In the case of $S^1 \times S^2$, the homeomorphism $h$ on the boundaries of the solid tori maps the meridian $m_1$ of one solid torus to the meridian $m_2$ of the other, as illustrated in Figure~\ref{stst}.

\[
\begin{matrix}
h & : & \partial {\rm ST}_1 & \longrightarrow & \partial {\rm ST}_2\\
  &   &   m_1                & \longmapsto     &  m_2
\end{matrix}
\]

\begin{figure}
\begin{center}
\includegraphics[width=3.7in]{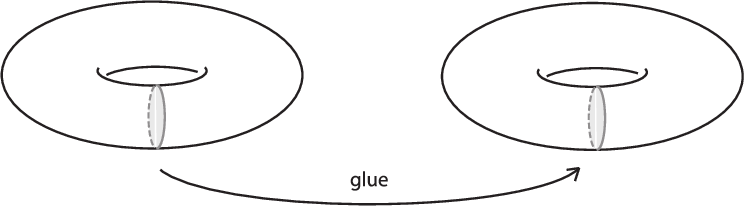}
\end{center}
\caption{Gluing two solid tori to obtain $S^1 \times S^2$.}
\label{stst}
\end{figure}

Thus, isotopy in $S^1 \times S^2$ can be viewed as isotopy in $\mathrm{ST}$ together with band moves in $S^3$, which reflect the surgery description of the manifold. Namely, two links in $S^1 \times S^2$ are isotopic if and only if any two corresponding mixed link diagrams of theirs differ by planar isotopy, a finite sequence of the Reidemeister moves for the standard part of the mixed link, the mixed Reidemeister moves that involve both the fixed and the moving part of the mixed links and which are illustrated in Figure~\ref{mr}, and the band moves, that reflect the surgery description of $S^1 \times S^2$, and which are illustrated in Figure~\ref{bmov1}.

\begin{figure}[H]
\begin{center}\includegraphics[width=3.6in]{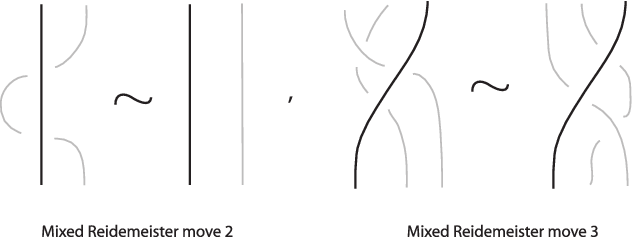}
\end{center}
\caption{The mixed Reidemeister moves.}
\label{mr}
\end{figure}

It is worth mentioning that there are two distinct types of band moves (depending on orientation), and that in \cite{DL1} it is shown that in order to describe isotopy for knots and links in a c.c.o. $3$-manifold, it suffices to consider only one type of band move. 

\begin{figure}[H]
\begin{center}
\includegraphics[width=5.5in]{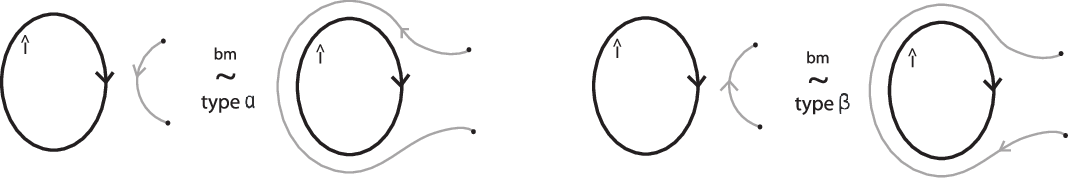}
\end{center}
\caption{The two types of band moves.}
\label{bmov1}
\end{figure}

Consequently, isotopy between oriented links in $S^1 \times S^2$ is reflected in $S^3$ by means of the following result (cf. Thm.~5.8 in \cite{LR1}, Thm.~6 in \cite{DL1}): 

\smallbreak

\begin{thm}
Two oriented links in $S^1 \times S^2$ are isotopic if and only if two corresponding mixed link diagrams of theirs differ by isotopy in $\mathrm{ST}$ together with a finite sequence of one type of band move.
\end{thm}

\subsection{Mixed braids and braid equivalence for knots and links in $L(0,1)$}

By the Alexander theorem for knots in the solid torus (cf. Thm.~1 in \cite{La2}), a mixed link diagram $\widehat{I}\cup \widetilde{L}$ of $\widehat{I}\cup L$ may be turned into a \textit{mixed braid} $I\cup \beta$ with isotopic closure (see Figure~\ref{cl}). This is a braid in $S^{3}$ where, without loss of generality, the first strand represents $\widehat{I}$ (the fixed part) and the other strands, $\beta$, represent the moving part $L$. The subbraid $\beta$ is called the \textit{moving part} of $I\cup \beta$ (see Figure~\ref{cl} and bottom left of Figure~\ref{bmov}).

\begin{figure}[H]
\begin{center}\includegraphics[width=2.5in]{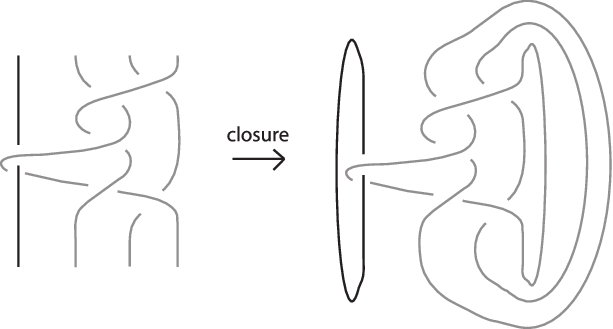}
\end{center}
\caption{The closure of a mixed braid to a mixed link.}
\label{cl}
\end{figure}

To translate isotopy for links in $S^1 \times S^2$ into braid equivalence, we first define a \textit{braid band move} (bbm) to be a move between mixed braids, corresponding to a band move between their closures. It starts with a small downward-oriented band which, before sliding along a surgery strand, acquires one twist (either \textit{positive} or \textit{negative}; see the bottom right-hand side of Figure~\ref{bmov}).

\begin{figure}[h]
\begin{center}
\includegraphics[width=3.3in]{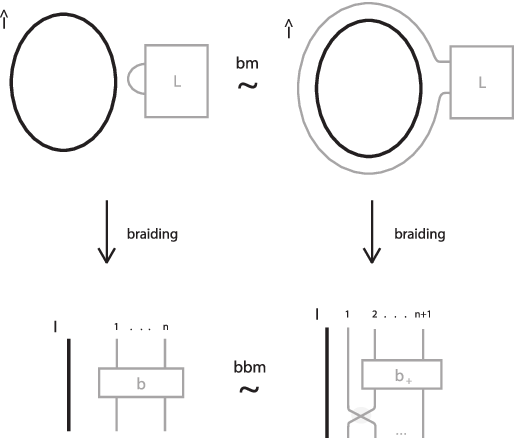}
\end{center}
\caption{Isotopy in $S^1 \times S^2$ and the two types of braid band moves on mixed braids.}
\label{bmov}
\end{figure}

\begin{remark}\rm
In \cite{LR2} it is shown that the position at which the two components are connected after performing a bbm is arbitrary.
\end{remark}

The sets of braids corresponding to $\mathrm{ST}$ form groups, namely the Artin braid groups of type B, denoted $B_{1,n}$, with presentation:

\[
B_{1,n} =
\left\langle
\begin{array}{ll}
t, \sigma_{1}, \ldots, \sigma_{n-1} &
\left|
\begin{array}{l}
\sigma_{1}t\sigma_{1}t = t\sigma_{1}t\sigma_{1} \\
t\sigma_{i} = \sigma_{i}t, \quad i>1 \\
\sigma_i\sigma_{i+1}\sigma_i = \sigma_{i+1}\sigma_i\sigma_{i+1}, \quad 1 \leq i \leq n-2 \\
\sigma_i\sigma_j = \sigma_j\sigma_i, \quad |i-j|>1
\end{array}
\right.
\end{array}
\right\rangle
\]

\noindent where the generators $t$ and $\sigma_i$ are illustrated in Figure~\ref{genh}.

\begin{figure}[H]
\begin{center}
\includegraphics[width=4.6in]{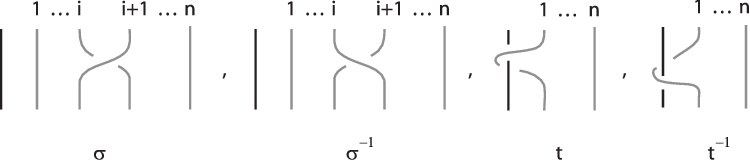}
\end{center}
\caption{The generators of \(B_{1,n}\).}
\label{genh}
\end{figure}

Isotopy in $S^1 \times S^2$ is then translated to the level of mixed braids via the following theorem (for an illustration see Figure!\ref{mbeq}):

\begin{thm}[Theorem~5, \cite{LR2}] \label{markov}
Let $L_1, L_2$ be two oriented links in $S^1 \times S^2$ and let $I\cup \beta_1$, $I\cup \beta_2$ be two corresponding mixed braids in $S^{3}$. Then $L_1$ is isotopic to $L_2$ in $S^1 \times S^2$ if and only if $I\cup \beta_1$ is equivalent to $I\cup \beta_2$ in $\mathcal{B}$ by the following moves:
\[
\begin{array}{clll}
(i)  & \text{Conjugation:}         & \alpha \sim \beta^{-1} \alpha \beta, & \text{for } \alpha, \beta \in B_{1,n}. \\
(ii) & \text{Stabilization moves:} & \alpha \sim \alpha \sigma_{n}^{\pm 1} \in B_{1,n+1}, & \text{for } \alpha \in B_{1,n}. \\
(iii) & \text{Loop conjugation:}   & \alpha \sim t^{\pm 1} \alpha t^{\mp 1}, & \text{for } \alpha \in B_{1,n}. \\
(iv) & \text{Braid band moves:}    & \alpha \sim  \alpha_+ \sigma_1^{\pm 1}, & \text{with } \alpha_+ \in B_{1,n+1},
\end{array}
\]

\noindent where $\alpha_+$ is the word $\alpha$ with all indices shifted by $+1$.
\end{thm}

Note that moves (i), (ii) and (iii) correspond to link isotopy in {\rm ST}.

\begin{figure}[H]
\begin{center}
\includegraphics[width=5.5in]{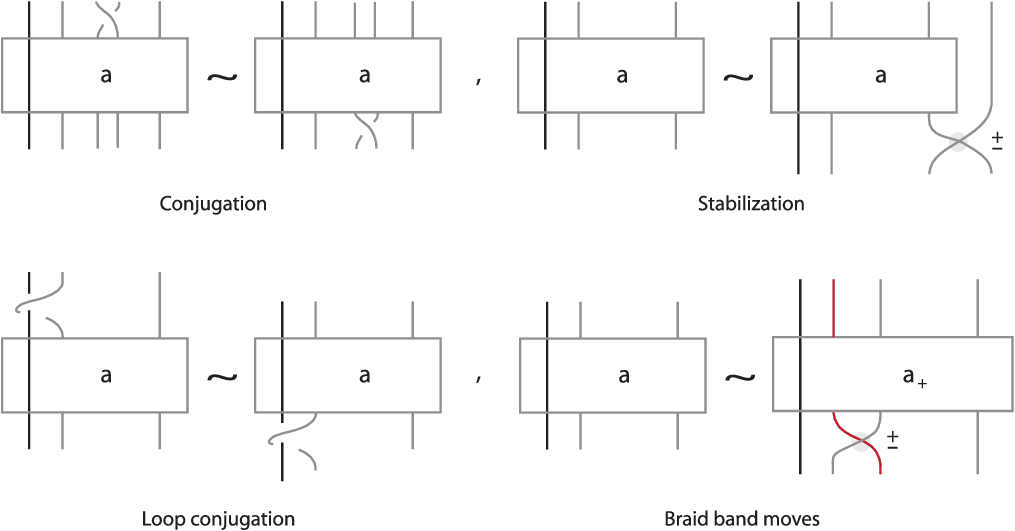}
\end{center}
\caption{Mixed braid equivalence.}
\label{mbeq}
\end{figure}

\begin{nt}\rm
We denote a braid band move by bbm and, specifically, the result of a positive or negative braid band move performed on the $i^{\text{th}}$ moving strand of a mixed braid $\beta$ by $bbm_{\pm i}(\beta)$.
\end{nt}

\subsection{The HOMFLYPT skein module of \(\mathrm{ST}\) via braids}\label{SolidTorus}

In \cite{HK} the HOMFLYPT skein module of the solid torus has been computed using diagrammatic methods by means of the following theorem:

\begin{thm}[Kidwell--Hoste] \label{turaev}
The HOMFLYPT skein module skein module of the solid torus, $\mathcal{S}({\rm ST})$, is a free, infinitely generated $\mathbb{Z}[u^{\pm1},z^{\pm1}]$-module isomorphic to the symmetric tensor algebra $SR\widehat{\pi}^0$, where $\widehat{\pi}^0$ denotes the conjugacy classes of non trivial elements of $\pi_1(\rm ST)$.
\end{thm}

A basic element of $\mathcal{S}({\rm ST})$ in the context of \cite{HK}, is illustrated in Figure~\ref{tur}. Note that in the diagrammatic setting of \cite{HK}, ST is considered as ${\rm Annulus} \times {\rm Interval}$. 

\begin{figure}[!ht]
\begin{center}
\includegraphics[width=1.4in]{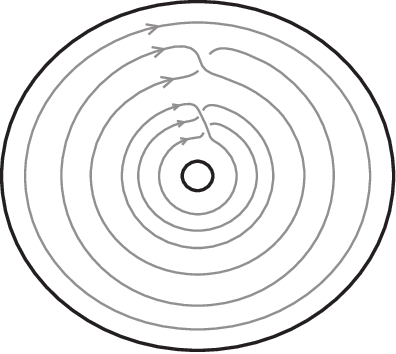}
\end{center}
\caption{A basic element of $\mathcal{S}({\rm ST})$.}
\label{tur}
\end{figure}

As already mentioned in the introduction, in \cite{La1} the most generic analogue of the HOMFLYPT polynomial, \(X\), for links in the solid torus \(\mathrm{ST}\), was derived from the generalized Iwahori--Hecke algebras of type \(\mathrm{B}\), \(\mathrm{H}_{1,n}(q)\), via a unique Markov trace constructed on them. This algebra was defined by Lambropoulou as the quotient of \({\mathbb C}\left[q^{\pm 1} \right]B_{1,n}\) over the quadratic relations \(g_{i}^2=(q-1)g_{i}+q\), corresponding to the HOMFLYPT skein relation. Namely:
\[
\mathrm{H}_{1,n}(q) =
\frac{{\mathbb C}\left[q^{\pm 1} \right]B_{1,n}}{\bigl\langle \sigma_i^2 - (q-1)\sigma_i -q \bigr\rangle}.
\]

Moreover, in \cite{La1} it is shown that the following sets form linear bases for \(\mathrm{H}_{1,n}(q)\) (\cite[Proposition~1 and Theorem~1]{La2}):
\begin{equation}\label{sigman}
\begin{array}{llll}
 (i) & \Sigma_{n} & = & \{t_{i_{1}}^{k_{1}} \ldots t_{i_{r}}^{k_{r}} \cdot \sigma \} ,\quad 0\le i_{1} <\ldots <i_{r} \le n-1, \\
 (ii) & \Sigma^{\prime}_{n} & = & \{{t^{\prime}_{i_1}}^{k_{1}} \ldots {t^{\prime}_{i_r}}^{k_{r}} \cdot \sigma \} ,\quad 0\le i_{1} <\ldots <i_{r} \le n-1, 
\end{array}
\end{equation}

\noindent where \(k_{1}, \ldots ,k_{r} \in {\mathbb Z}\), \(t_0^{\prime} = t_0 := t\), \(t_i^{\prime} = g_i\cdots g_1 t g_1^{-1}\cdots g_i^{-1}\), and \(t_i = g_i\cdots g_1 t g_1\cdots g_i\) are the so-called `looping elements' of \(\mathrm{H}_{1,n}(q)\) (see Figure~\ref{lp}). Here \(\sigma\) is a basis element of the Iwahori--Hecke algebra of type A, \(\mathrm{H}_n(q)\), which, as shown in \cite{Jo}, is of the form:
\[
S =
\bigl\{
(g_{i_1}g_{i_1-1}\cdots g_{i_1-k_1})
(g_{i_2}g_{i_2-1}\cdots g_{i_2-k_2})
\cdots
(g_{i_p}g_{i_p-1}\cdots g_{i_p-k_p})
\bigr\},
\]
for \(1\le i_1 < \cdots < i_p \le n-1\).

\begin{figure}[H]
\begin{center}
\includegraphics[width=2.7in]{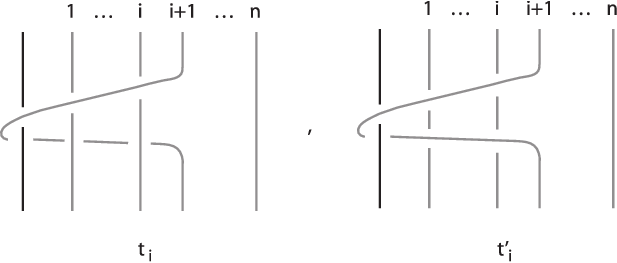}
\end{center}
\caption{The looping generators.}
\label{lp}
\end{figure}

In \cite{La1}, the bases \(\Sigma^{\prime}_{n}\) are used to construct a Markov trace on \(\mathcal{H}:=\bigcup_{n=1}^{\infty} \mathrm{H}_{1,n}(q)\). More precisely:

\begin{thm}[\cite{La1}, Theorem~6] \label{tr}
Given specified elements \(z, s_k\), \(k\in {\mathbb Z}\), in \(R={\mathbb Z}\bigl[q^{\pm 1}\bigr]\), there exists a unique linear Markov trace function
\[
\mathrm{tr} : \mathcal{H} \longrightarrow R(z,s_k), \quad k\in {\mathbb Z},
\]
determined by:
\[
\begin{array}{llll}
(1) & \mathrm{tr}(ab) = \mathrm{tr}(ba) & \text{for all } a,b\in \mathrm{H}_{1,n}(q) \\
(2) & \mathrm{tr}(1) = 1 & \text{in each } \mathrm{H}_{1,n}(q) \\
(3) & \mathrm{tr}(a g_n) = z\,\mathrm{tr}(a) & \text{for } a\in \mathrm{H}_{1,n}(q) \\
(4) & \mathrm{tr}\bigl( a (t'_n)^k \bigr) = s_k\,\mathrm{tr}(a) & \text{for } a\in \mathrm{H}_{1,n}(q),\ k\in {\mathbb Z}.
\end{array}
\]
\end{thm}

Using this Markov trace, Lambropoulou constructed a universal HOMFLYPT-type invariant for oriented links in \(\mathrm{ST}\). Namely, let \(\mathcal{L}\) denote the set of oriented links in \(\mathrm{ST}\). Then:

\begin{thm}[\cite{La1}, Definition~1] \label{inv}
The function \(X : \mathcal{L} \to R(z,s_k)\) given by
\[
X_{\widehat{\alpha}} =
\Delta^{n-1} \bigl( \sqrt{\lambda} \bigr)^e\,
\mathrm{tr}\bigl( \pi(\alpha) \bigr),
\]
where \(\Delta := -\frac{1-\lambda q}{\sqrt{\lambda}(1-q)}\), \(\lambda := \frac{z+1-q}{qz}\), \(\alpha \in B_{1,n}\) is a word in the generators \(\sigma_i\) and \(t'_i\), \(\widehat{\alpha}\) is the closure of \(\alpha\), \(e\) is the exponent sum of the \(\sigma_i\) in \(\alpha\), and \(\pi\) is the canonical map \(B_{1,n} \to \mathrm{H}_{1,n}(q)\) defined by \(t\mapsto t\), \(\sigma_i\mapsto g_i\), is an invariant of oriented links in \(\mathrm{ST}\).
\end{thm}

\smallbreak

In the braid setting of \cite{La1}, the elements of \(\mathcal{S}(\mathrm{ST})\) correspond bijectively to the elements of the following set:
\begin{equation}\label{Lpr}
\Lambda' =
\bigl\{
t^{k_0} {t'_1}^{k_1} \cdots {t'_n}^{k_n} :
k_i\in\mathbb{Z}\setminus\{0\},\ k_i\geq k_{i+1}\ \forall i,\ n\in\mathbb{N}
\bigr\}.
\end{equation}

\noindent As explained in \cite{DL2}, \(\Lambda'\) forms a basis of \(\mathcal{S}(\mathrm{ST})\) in terms of braids and corresponds to the basis of \cite{HK}. Note that \(\Lambda'\) is a subset of \(\mathcal{H}\), and in particular of \(\Sigma'\). Unlike elements of \(\Sigma'\), the elements of \(\Lambda'\) have no gaps in the indices, their exponents are ordered, and they have no `braiding tails'.

\begin{remark}\rm
The Lambropoulou invariant \(X\) recovers \(\mathcal{S}(\mathrm{ST})\), as it assigns distinct values to distinct elements of \(\Lambda'\), since \(\mathrm{tr}\bigl(t^{k_0} {t'_1}^{k_1} \cdots {t'_n}^{k_n}\bigr) = s_{k_n}\cdots s_{k_1}s_{k_0}\).
\end{remark}

\section{A different basis \(\Lambda\) of \(\mathcal{S}({\rm ST})\)}\label{newbas}

In this section, we recall results from \cite{DL2, DLP, DL4} that are instrumental in the computation of \(\mathcal{S}(S^1\times S^2)\). For a survey on HOMFLYPT skein modules via braids, the reader is referred to \cite{DL3, DAMS, DGLM}.

\bigbreak

In \cite{DL2}, a different basis \(\Lambda\) for \(\mathcal{S}({\rm ST})\) was introduced, which is crucial for the computation of \(\mathcal{S}(S^1 \times S^2)\). This basis is described in Eq.~(\ref{basis}) in open braid form:

\begin{thm}[\cite{DL2}, Theorem~2] \label{newbasis}
The following set is a \(\mathbb{C}[q^{\pm1}, z^{\pm1}]\)-basis for \(\mathcal{S}({\rm ST})\):
\begin{equation}\label{basis}
\Lambda =
\bigl\{
t^{k_0} t_1^{k_1} \cdots t_n^{k_n} \,\big|\, k_i \in \mathbb{Z}\setminus\{0\},\ k_i \geq k_{i+1} \text{ for all } i,\ n \in \mathbb{N}
\bigr\}.
\end{equation}
\end{thm}

The importance of the basis \(\Lambda\) lies in the fact that braid band moves, which extend link isotopy in \(\mathrm{ST}\) to link isotopy in \(S^1 \times S^2\), have a much simpler algebraic expression in this setting. This was the main motivation for introducing \(\Lambda\) (recall Theorem~\ref{markov}(iv)).

\begin{figure}[h]
\begin{center}
\includegraphics[width=1.8in]{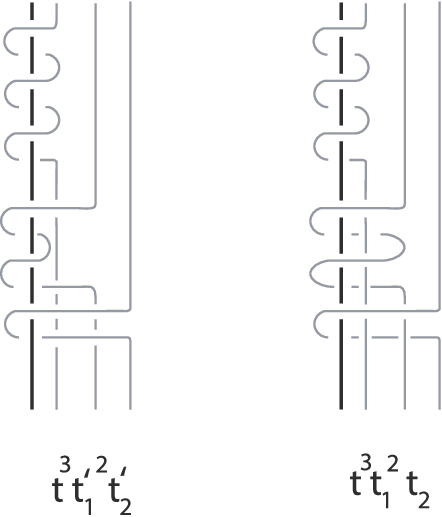}
\end{center}
\caption{Elements in the two different bases of \(\mathcal{S}({\rm ST})\).}
\label{basel}
\end{figure}

Comparing \(\Lambda\) to the linear basis \(\Sigma\) of \({\rm H}_{1,n}(q)\) (recall Eq.~\ref{sigman}), we observe that in \(\Lambda\) there are no gaps in the indices of the \(t_i\)'s, the exponents are in decreasing order, and no `braiding tails' are present.

\bigbreak

The proof of Theorem~\ref{newbasis} in \cite{DL2} proceeds as follows:
\begin{itemize}
\item[-] Define total orderings on the sets \(\Lambda'\) and \(\Lambda\);
\item[-] Show that these ordered sets are related via a lower-triangular infinite matrix with invertible diagonal entries;
\item[-] Conclude linear independence of \(\Lambda\) using this matrix.
\end{itemize}

More precisely, to relate \(\Lambda'\) and \(\Lambda\) via this matrix, one starts with elements of the standard basis \(\Lambda'\) and rewrites them as sums of elements in the linear bases \(\Sigma_n\) of \({\rm H}_{1,n}(q)\), that is, products of arbitrary monomials in \(t_i\)'s followed by `braiding tails' in \({\rm H}_n(q)\). These are then converted into elements of \(\Lambda\) through the following steps:
\begin{itemize}
\item[-] managing gaps in the indices of the \(t_i\)'s,
\item[-] ordering the exponents of the \(t_i\)'s,
\item[-] eliminating the `braiding tails'.
\end{itemize}

It is worth noting that these procedures are interdependent: for instance, managing gaps may introduce braiding tails or alter the exponents, and eliminating braiding tails may create gaps. Nevertheless, as shown in \cite{DL2}, this iterative process eventually terminates, leaving elements in the desired form of \(\Lambda\).

\subsection{An ordering in the bases of \(\mathcal{S}({\rm ST})\)}

In \cite{DL2}, an ordering relation is defined on the sets \(\Sigma\), \(\Sigma^{\prime}\), \(\Lambda\), and \(\Lambda^{\prime}\), which plays an important role in this paper. Before presenting this ordering, we introduce the notion of the \textit{index} of a word \(w\), denoted \(ind(w)\), in any of these sets.

\begin{defn}[\cite{DL2}, Definition~1] \rm 
Let \(w\) be a word in \(\Lambda^{\prime}\) or \(\Lambda\). Then \(ind(w)\) is defined as the highest index among the \(t_i^{\prime}\)'s (respectively \(t_i\)'s) appearing in \(w\). Similarly, for \(w\) in \(\Sigma^{\prime}\) or \(\Sigma\), \(ind(w)\) is defined by ignoring possible gaps in the indices of the looping generators and by ignoring the braiding parts in the algebras \(\mathrm{H}_n(q)\). Moreover, the index of a monomial in \(\mathrm{H}_n(q)\) is equal to \(0\).
\end{defn}

We now present the ordering relation on \(\Sigma^{\prime}\), which induces corresponding orderings on its subsets \(\Lambda\) and \(\Lambda^{\prime}\).

\begin{defn}[\cite{DL2}, Definition~2] \label{order} \rm
Let \(w={t^{\prime}_{i_1}}^{k_1}\cdots {t^{\prime}_{i_{\mu}}}^{k_{\mu}} \cdot \beta_1\) and \(u={t^{\prime}_{j_1}}^{\lambda_1}\cdots {t^{\prime}_{j_{\nu}}}^{\lambda_{\nu}} \cdot \beta_2\) be elements of \(\Sigma^{\prime}\), where \(k_t, \lambda_s \in \mathbb{Z}\) for all \(t,s\), and \(\beta_1, \beta_2 \in H_n(q)\). Then we define \(w < u\) if any of the following hold:

\begin{itemize}
\item[(a)] \(\sum_{i=0}^{\mu} k_i < \sum_{i=0}^{\nu} \lambda_i\);

\item[(b)] If \(\sum_{i=0}^{\mu} k_i = \sum_{i=0}^{\nu} \lambda_i\), then:
\begin{itemize}
\item[(i)] if \(ind(w) < ind(u)\), then \(w < u\);
\item[(ii)] if \(ind(w) = ind(u)\), then:
\begin{itemize}
\item[($\alpha$)] if \(i_1=j_1, \ldots, i_{s-1}=j_{s-1}, i_s<j_s\), then \(w>u\);
\item[($\beta$)] if \(i_t=j_t\) for all \(t\) and \(k_{\mu}=\lambda_{\mu}, \ldots, k_{i+1}=\lambda_{i+1}, |k_i|<|\lambda_i|\), then \(w<u\);
\item[($\gamma$)] if \(i_t=j_t\) for all \(t\) and \(k_{\mu}=\lambda_{\mu}, \ldots, k_{i+1}=\lambda_{i+1}, |k_i|=|\lambda_i|\) and \(k_i>\lambda_i\), then \(w<u\);
\item[($\delta$)] if \(i_t=j_t\) and \(k_i=\lambda_i\) for all \(i\), then \(w=u\).
\end{itemize}
\end{itemize}
\end{itemize}

The ordering on \(\Sigma\) is defined analogously, replacing \(t_i^{\prime}\) with \(t_i\).
\end{defn}

\bigbreak

We also define the \textit{subsets of level \(k\)}, \(\Lambda_{(k)}\) and \(\Lambda^{\prime}_{(k)}\), of \(\Lambda\) and \(\Lambda^{\prime}\) respectively (\cite{DL2}, Definition~3):

\begin{equation}
\begin{array}{l}
\Lambda_{(k)} := \bigl\{t_0^{k_0} t_1^{k_1} \cdots t_m^{k_m} \,\big|\, \sum_{i=0}^m k_i = k,\ k_i \neq 0,\ k_i \geq k_{i+1} \ \forall i \bigr\} \\[1ex]
\Lambda^{\prime}_{(k)} := \bigl\{{t^{\prime}_0}^{k_0} {t^{\prime}_1}^{k_1} \cdots {t^{\prime}_m}^{k_m} \,\big|\, \sum_{i=0}^m k_i = k,\ k_i \neq 0,\ k_i \geq k_{i+1} \ \forall i \bigr\}.
\end{array}
\end{equation}

\noindent In \cite{DL2}, it is shown that the sets \(\Lambda_{(k)}\) and \(\Lambda^{\prime}_{(k)}\) are totally ordered and well-ordered for all \(k\) (\cite[Propositions~1 \& 2]{DL2}).

\bigbreak

Finally, we define the augmented set \(\Lambda^{\text{aug}}\) and its level-\(k\) subsets \(\Lambda^{\text{aug}}_{(k)}\):

\begin{equation}\label{lamaug}
\Lambda^{\text{aug}} := \bigl\{t_0^{k_0} t_1^{k_1} \cdots t_n^{k_n} \,\big|\, k_i \neq 0 \bigr\},
\end{equation}

\begin{equation}\label{lamaug2}
\Lambda^{\text{aug}}_{(k)} := \bigl\{t_0^{k_0} t_1^{k_1} \cdots t_m^{k_m} \,\big|\, \sum_{i=0}^m k_i = k,\ k_i \neq 0 \bigr\}.
\end{equation}

These sets, introduced in \cite{DLP}, play a key role in the computation of \(\mathcal{S}(L(p,1))\) in \cite{DL4} and \cite{D1}.

\subsection{The infinite change of basis matrix: from \(\Lambda^{\prime}\) to \(\Lambda\)}

In order to relate the two sets \(\Lambda^{\prime}\) and \(\Lambda\) via an infinite lower-triangular matrix, we start by converting monomials in the \(t^{\prime}_i\)'s into expressions involving the \(t_i\)'s. To simplify these expressions, we first introduce the following notation:

\begin{nt}\label{nt} \rm
We set \(\tau_{i,i+m}^{k_{i,i+m}} := t_i^{k_i} t_{i+1}^{k_{i+1}} \cdots t_{i+m}^{k_{i+m}}\) and \({\tau^{\prime}}_{i,i+m}^{k_{i,i+m}} := {t^{\prime}}_i^{k_i} {t^{\prime}}_{i+1}^{k_{i+1}} \cdots {t^{\prime}}_{i+m}^{k_{i+m}}\), where \(m \in \mathbb{N}\) and \(k_j \neq 0\) for all \(j\).
\end{nt}

We now introduce the notion of \textit{homologous words}, which is crucial for expressing elements of \(\Lambda^{\prime}\) in terms of \(\Lambda\) through a triangular matrix.

\begin{defn}[\cite{DL2}, Definition~4] \rm
We say that two words \(w^{\prime} \in \Lambda^{\prime}\) and \(w \in \Lambda\) are \textit{homologous}, denoted \(w^{\prime} \sim w\), if \(w\) is obtained from \(w^{\prime}\) by replacing each \(t^{\prime}_i\) with \(t_i\), for all \(i\).
\end{defn}

As shown in \cite{DL2}, every element of \(\Lambda^{\prime}\) can be expressed as a sum of elements of \(\Sigma\) of lower order than the homologous word (see \cite[Theorem~7]{DL2}). More precisely:

\[
{\tau^{\prime}}_{0,m}^{k_{0,m}}  =
q^{- \sum_{n=1}^{m} n k_n} \cdot {\tau}_{0,m}^{k_{0,m}}
\ +\ \sum_{i} f_i(q) \cdot {\tau}_{0,m}^{k_{0,m}} w_i
\ +\ \sum_{j} g_j(q) \tau_j u_j,
\]

where \(w_i, u_j \in \mathrm{H}_{m+1}(q)\), each \(\tau_j\) is a monomial in the \(t_i\)'s such that \(\tau_j < \tau_{0,m}^{k_{0,m}}\), and \(f_i, g_j \in \mathbb{C}\).

\smallbreak

When converting \(\tau^{\prime} \in \Lambda^{\prime}\) into a sum of elements of \(\Sigma\), one obtains the homologous word \(\tau\) together with terms that are not elements of \(\Lambda\) (e.g., terms with gaps in the indices, unordered exponents, or braiding parts). As shown in \cite{DL2}, by applying conjugation and stabilization moves, these terms can be further reduced to sums of elements of \(\Lambda\) of strictly lower order.

We use the notation \(\widehat{=}\) to denote an equality up to conjugation, \(\simeq\) for stabilization, and \(\widehat{\simeq}\) when both conjugation and stabilization have been applied.

\begin{thm}[\cite{DL2}, Theorems~7, 8, 9, 10] \label{gp}
Let \(\tau^{\prime} \in \Lambda^{\prime}\). Then, applying conjugation and stabilization moves yields:
\[
\tau^{\prime} \ \widehat{\simeq} \ a \cdot \tau\ +\ \sum_{i} a_i \cdot \tau_i,
\]
where \(\tau \sim \tau^{\prime}\), each \(\tau_i \in \Lambda\) satisfies \(\tau_i < \tau\), and \(a, a_i \in \mathbb{C}\).
\end{thm}

With the ordering defined in Definition~\ref{order}, it follows that the infinite matrix converting elements of \(\Lambda^{\prime}\) into elements of \(\Lambda\) is block diagonal, where each block is an infinite lower-triangular matrix with invertible diagonal entries. Denoting by \(M_k\) the block matrix acting on the level-\(k\) components \(\Lambda^{\prime}_{(k)}\) and \(\Lambda_{(k)}\), the full change of basis matrix \(M\) has the form:
\[
M =
\begin{bmatrix}
\ddots & 0 & 0 & 0 & 0 & 0 & \\
& M_{k-2} & 0 & 0 & 0 & 0 & \\
& 0 & M_{k-1} & 0 & 0 & 0 & \\
& 0 & 0 & M_k & 0 & 0 & \\
& 0 & 0 & 0 & M_{k+1} & 0 & \\
& 0 & 0 & 0 & 0 & M_{k+2} & \\
& 0 & 0 & 0 & 0 & 0 & \ddots
\end{bmatrix}
\]

\begin{center}
\textit{The infinite block-diagonal change of basis matrix}
\end{center}

\begin{remark}\label{rem2}\rm
\begin{itemize}
\item[(i)] The inverse of \(M\), \(M^{-1}\), is also block diagonal, with each block being an infinite upper-triangular matrix with invertible diagonal entries. Therefore, an element \(\tau^{\prime} \in \Lambda^{\prime}\) can be expressed as a linear combination of its homologous word \(\tau \in \Lambda\) and elements of \(\Lambda\) of lower order.
\item[(ii)] Combining only Theorems~7, 8, and 10 of \cite{DL2} shows that \(\tau^{\prime} \in \Lambda^{\prime}\) can be written as a sum of elements of \(\Lambda^{\text{aug}}\), that is, monomials in the \(t_i\)'s without gaps in the indices and with possibly unordered exponents.
\end{itemize}
\end{remark}

\section{The HOMFLYPT skein module of \(S^1 \times S^2\) via braids}\label{HOMS}

In this section, we compute the HOMFLYPT skein module of \(S^1 \times S^2\) using braid techniques. Specifically, we solve the infinite system of equations obtained by applying braid band moves (bbm's) on the first moving strand of elements in the augmented set \(\Lambda^{\text{aug}}\).

\subsection{Relating \(\mathcal{S}(S^1\times S^2)\) to \(\mathcal{S}({\rm ST})\)}

As explained in \cite{DLP}, to compute \(\mathcal{S}(L(p, 1))\), the invariant \(X\) must be normalized to satisfy all possible braid band moves. In particular, \cite{DLP} shows that it suffices to consider elements of the basic set \(\Lambda\) and to perform bbm's on all their moving strands. That is:
\[
X_{\widehat{\tau}} = X_{\widehat{bbm_i(\tau)}}, \quad \forall\, \tau \in \Lambda,
\]
where \(bbm_i(\tau)\) denotes the result of performing a bbm on the \(i^\text{th}\) moving strand of \(\tau \in \Lambda\).

\smallskip

In \cite{DL4} it is shown that, in fact, it suffices to perform bbm's only on the \textit{first} moving strand of elements of the augmented set \(\Lambda^{\text{aug}}\). The methods of \cite{DLP, DL4} apply equally well in the case \(L(0,1) \cong S^1 \times S^2\), leading to the following result:

\begin{thm}\label{dlnew}
The solution of the infinite system
\[
X_{\widehat{\tau}} = X_{\widehat{bbm_1(\tau)}}, \quad \forall\, \tau \in \Lambda^{\text{aug}},
\]
where \(bbm_1(\tau)\) denotes the result of performing a bbm on the first moving strand of \(\tau \in \Lambda^{\text{aug}}\), is equivalent to finding a basis for the HOMFLYPT skein module of \(S^1 \times S^2\).
\end{thm}

The proof of Theorem~\ref{dlnew} follows the same steps as in the case of \(L(p,1)\), with the only difference being the form of the bbm's, which do not affect the overall argument (see \cite{DLP, DL4} for details). The steps are summarized in the following diagram:
\[
\begin{array}{llllll}
\mathcal{S}(S^1 \times S^2) & = & \displaystyle\frac{\mathcal{S}({\rm ST})}{\langle a - bbm_i(a) \rangle},\ a\in B_{1,n},\ \forall\, i
& = & \displaystyle\frac{\mathcal{S}({\rm ST})}{\langle s^{\prime} - bbm_1(s^{\prime}) \rangle},\ s^{\prime}\in \Sigma_n^{\prime} \\[12pt]

& = & \displaystyle\frac{\mathcal{S}({\rm ST})}{\langle s - bbm_1(s) \rangle},\ s\in \Sigma_n
& = & \displaystyle\frac{\mathcal{S}({\rm ST})}{\langle \lambda^{\prime} - bbm_i(\lambda^{\prime}) \rangle},\ \lambda^{\prime} \in \Lambda^{\text{aug}}|{\rm H}_n,\ \forall\, i \\[12pt]

& = & \displaystyle\frac{\mathcal{S}({\rm ST})}{\langle \lambda^{\prime\prime} - bbm_i(\lambda^{\prime\prime}) \rangle},\ \lambda^{\prime\prime}\in \Lambda|{\rm H}_n,\ \forall\, i
& = & \displaystyle\frac{\mathcal{S}({\rm ST})}{\langle \tau - bbm_1(\tau) \rangle},\ \tau\in \Lambda^{\text{aug}}.
\end{array}
\]

\noindent Here, \(\Lambda^{\text{aug}}|{\rm H}_n\) denotes the \({\rm H}_n(q)\)-module \(\Lambda^{\text{aug}}\), consisting of monomials in the \(t_i\)'s (without gaps in indices and possibly unordered exponents) followed by `braiding tails' in \({\rm H}_n\). The difference between \(\Lambda^{\text{aug}}|{\rm H}_n\) and \(\Lambda|{\rm H}_n\) is that in the latter the exponents are ordered.

\subsection{The infinite system}\label{sys}

We now simplify the equations in the infinite system and show that the unknowns commute. The key result of this subsection is that the system splits into infinitely many, self-contained subsystems.

\begin{lemma}
Let \(\tau_{0,m}^{k_{0,m}} \in \Lambda^{\text{aug}}_k\), with \(\sum_{i=0}^{m}k_i=k\). Then the system
\[
\left\{
\begin{array}{lll}
X_{\widehat{\tau_{0,m}^{k_{0,m}}}} & = & X_{\widehat{\tau_{1,m+1}^{k_{0,m}}\sigma_1}} \\[6pt]
X_{\widehat{\tau_{0,m}^{k_{0,m}}}} & = & X_{\widehat{\tau_{1,m+1}^{k_{0,m}}\sigma_1^{-1}}}
\end{array}
\right.
\]
is equivalent to
\[
\begin{array}{lll}
\mathrm{tr}(\tau_{0,m}^{k_{0,m}}) & = & \displaystyle\frac{1}{z}\cdot \sqrt{\lambda}^{\,e_2-e_1-1}\cdot \mathrm{tr}(\tau_{1,m+1}^{k_{0,m}}\cdot \sigma_1), \\[6pt]
\mathrm{tr}(\tau_{0,m}^{k_{0,m}}) & = & \displaystyle\frac{1}{z}\cdot \sqrt{\lambda}^{\,e_2-3}\cdot \mathrm{tr}(\tau_{1,m+1}^{k_{0,m}}\cdot \sigma_1^{-1}),
\end{array}
\]
where \(e_1=\sum_{i=1}^{m}2i k_i\) and \(e_2=\sum_{i=1}^{m+1}(2i\,k_{i-1}+1)\).
\end{lemma}

\begin{proof}
The proof follows by computing \(X_{\widehat{\tau}}\) for \(\tau=\tau_{0,m}^{k_{0,m}}\), performing a bbm on the first moving strand, and comparing exponents of \(\sqrt{\lambda}\) and \(z\). Details are identical to the analogous computation for \(L(p,1)\) in \cite{DL4}, but we include them below since the case $p=0$ has some subtle differences worth clarifying.

Let $\tau=\tau_{0,m}^{k_{0,m}}$. Then $X_{\widehat{\tau}}=\left(\frac{1}{\sqrt{\lambda}\cdot z} \right)^m\cdot \sqrt{\lambda}^{e_1}\cdot tr(\tau)$, where $e_1=\sum_{i=1}^{m}{2\cdot i\cdot k_i}$. We perform a {\it bbm} on the first moving strand of $\tau$ and obtain $\tau_+\cdot \sigma_1^{\pm 1}$, where $\tau_{+}=\tau_{1, m+1}^{k_{0,m}}:=t_1^{k_0}t_1^{k_1}\ldots t_{m+1}^{k_m}$. Moreover:

\[
\begin{array}{llll}
X_{\widehat{\tau_+\cdot \sigma_1}} & = & \left(\frac{1}{\sqrt{\lambda}\cdot z} \right)^{m+1}\cdot \sqrt{\lambda}^{\ e_2}\cdot tr(\tau_+\cdot \sigma_1), &\ \ \ e_2= \sum_{i=1}^{m+1}{2\cdot i\cdot k_{i-1}+1}  \\
X_{\widehat{\tau_+\cdot \sigma_1^{-1}}} & = & \left(\frac{1}{\sqrt{\lambda}\cdot z} \right)^{m+1}\cdot \sqrt{\lambda}^{\ e_3}\cdot tr(\tau_+\cdot \sigma_1^{-1}), &\ \ \ e_3= e_2-2
\end{array}
\]

\[
\begin{array}{lllll}
X_{\widehat{\tau}} & = & X_{\widehat{\tau_+\cdot \sigma_1}} & \Rightarrow &  tr(\tau)=\left(\frac{1}{z} \right)\cdot \sqrt{\lambda}^{\ e_2-e_1-1}\cdot tr(\tau_+\cdot \sigma_1)\\
X_{\widehat{\tau}} & = & X_{\widehat{\tau_+\cdot \sigma_1^{-1}}} & \Rightarrow & tr(\tau)=\left(\frac{1}{z} \right)\cdot \sqrt{\lambda}^{\ e_3-e_1-3}\cdot tr(\tau_+\cdot \sigma_1^{-1})
\end{array}
\]

\end{proof}

As shown in \cite{DL4}, the unknowns \(s_1, s_2, \ldots\) of the infinite system commute.

\begin{prop}
The unknowns \(s_1, s_2, \ldots\) of the infinite system of equations commute.
\end{prop}

We now introduce convenient notation and the relevant subsets of monomials:

\begin{nt}\rm
We denote by \(s_{m,l}^{\epsilon_{m,l}} = s_{m}^{\epsilon_m}s_{m+1}^{\epsilon_{m+1}}\cdots s_l^{\epsilon_l}\), where \(m<l\), and set \(s_0:=1\). We also use \(s_{-k_m,k_n}^{\epsilon_{-k_m,k_n}} = s_{-k_m}^{\epsilon_{-k_m}}\cdots s_{-1}^{\epsilon_{-1}} s_1^{\epsilon_1}\cdots s_{k_n}^{\epsilon_{k_n}}\), where \(\epsilon_i\in\{0,1\}\).
\end{nt}

\begin{defn}\rm \label{Ssets}
Define the sets
\[
S := \{ s_{m,l}^{\epsilon_{m,l}} \mid m,l\in\mathbb{Z},\, \epsilon_i\in\{0,1\} \ \forall i \}, \quad
S_{(k)} := \{ s_{m,l}^{\epsilon_{m,l}} \in S \mid \textstyle\sum_{i\in I}i\cdot\epsilon_i=k \},
\]
where \(I\) is the set of indices appearing in \(s_{m,l}^{\epsilon_{m,l}}\).
\end{defn}

We now establish the splitting of the infinite system into self-contained subsystems:

\begin{thm}\label{comb}
Let \(\tau_{0,m}^{k_{0,m}}\in \Lambda_{(k)}\). Then
\[
\mathrm{tr}(\tau_{0,m}^{k_{0,m}})=\sum_i A_i\,s_{j,i}^{\epsilon_{j,i}},
\]
where \(s_{j,i}^{\epsilon_{j,i}} \in S_{(k)}\) and \(A_i \in \mathbb{C}[q^{\pm1},z^{\pm1}]\).
\end{thm}

\begin{proof}

Let $\tau \in \Lambda_{(k)}$. Then, using the inverse of the change of basis matrix, we express $\tau$ as sum of elements in $\Lambda^{\prime}$ and we obtain:

\[
\tau \ \widehat{\cong}\ \underset{i}{\sum} {B_i\cdot \tau_i^{\prime}}\ +\ B\cdot \tau^{\prime},
\]

\noindent where $B, B_i\in \mathbb{C}[q^{\pm 1}, z^{\pm 1}]$ and $\tau^{\prime}, \tau_i^{\prime} \in \Lambda_k^{\prime}$ such that $\tau \sim \tau^{\prime} < \tau_i^{\prime}, \forall i$. Moreover, we have that
\smallbreak
$${\rm tr}(\tau)\ =\ {\rm tr}\left(\underset{i}{\sum} {B_i\cdot \tau_i^{\prime}}\ +\ B\cdot \tau^{\prime}\right)\ =\ \underset{i}{\sum} {B_i\cdot {\rm tr}(\tau_i^{\prime})}\ +\ B\cdot {\rm tr}(\tau^{\prime}),$$

\noindent and the result follows from the fourth rule of the trace.

\end{proof}

\begin{cor}
The infinite system of equations splits into infinitely many \emph{self-contained} subsystems, each corresponding to a fixed level \(k\).
\end{cor}

It is worth mentioning that each subsystem arises from bbm's applied to elements of \(\Lambda^{\text{aug}}_{(k)}\), or equivalently from bbm's on all moving strands of elements of \(\Lambda_{(k)}\) for a fixed \(k\in\mathbb{Z}\).

\subsection{An ordering in the set S}

We now define an {\it ordering} relation in the sets $S$ and $S_{(k)}$. We first need the following:

\begin{defn}\rm
Let $s_{m_i, m_j}^{\epsilon_{m_i, m_j}} \in S$.
\begin{itemize}
\item[i.] We define $e(s_{m_i, m_j}^{\epsilon_{m_i, m_j}})$ to be the sum of the exponents in the monomial $s_{m_i, m_j}^{\epsilon_{m_i, m_j}}$, namely:
$$e(s_{m_i, m_j}^{\epsilon_{m_i, m_j}})\ :=\ \underset{c \in I}{\sum}{\epsilon_c}.$$
\item[ii.] We define $in(s_{m_i, m_j}^{\epsilon_{m_i, m_j}})$ to be the set consisting of the indices of $s_{m_i, m_j}^{\epsilon_{m_i, m_j}} \in S$, namely:
$$in(s_{m_i, m_j}^{\epsilon_{m_i, m_j}})\ =\ \{m_c \in \mathbb{Z}\ |\ \epsilon_c=1\}.$$
\item[iii.] Finally, we define $in(|s_{m_i, m_j}^{\epsilon_{m_i, m_j}}|)$ to be the set consisting of the absolute values of the indices of $s_{m_i, m_j}^{\epsilon_{m_i, m_j}} \in S$, namely:
$$in(|s_{m_i, m_j}^{\epsilon_{m_i, m_j}}|)\ =\ \{m_c \in \mathbb{N}\ |\ \epsilon_c=1\}.$$
\end{itemize}
\end{defn}

\begin{ex}\rm
Let $s_{-2}s_1s_1s_3 \in S_{3}$. Then: $e(s_{-2}s_1s_1s_3) = 4$, $in(s_{-2}s_1s_1s_3) = \{-2, 1, 3 \}$ and $in(|s_{-2}s_1s_1s_3|) = \{1, 2, 3 \}$.
\end{ex}

\begin{defn}\label{ord2}\rm
Let $u_1, u_2 \in S$ such that:

$$u_1=s_{-k_m}^{\epsilon_{-k_m}}s_{-k_{m}+1}^{\epsilon_{-k_{m}+1}} \ldots s_{-1}^{\epsilon_{-1}} s_1^{\epsilon_1}s_2^{\epsilon_2}\ldots s_{k_n}^{\epsilon_{k_{n}}}\quad {\rm and} \quad u_2=s_{-l_a}^{\epsilon_{-l_a}}s_{-l_a+1}^{\epsilon_{-l_a+1}} \ldots s_{-1}^{\epsilon_{-1}} s_1^{\epsilon_1}s_2^{\epsilon_2}\ldots s_{l_b}^{\epsilon_{l_{b}}},$$ 

\noindent where $m, n, a, b \in \mathbb{N}$ and $k_i, l_i \in \mathbb{Z}$ and $\epsilon_i \in \{0, 1\}$ for all $i$. Then:

\smallbreak

\begin{itemize}
\item[(a)] If $u_1 \in S_{\mu}$ and $u_2 \in S_{\nu}$ such that $\mu < \nu$, then $u_1<u_2$.

\vspace{.1in}

\item[(b)] If $u_1, u_2 \in S_{\mu}$ and  $e(u_1) < e(u_2)$, then $u_1 < u_2$.

\vspace{.1in}

\item[(c)] If $u_1, u_2 \in S_{\mu}$, $e(u_1) = e(u_2)$ and $\max\{in(|u_1|) \} < \max\{in(|u_2|)\}$, then $u_1<u_2$.

\vspace{.1in}

\item[(d)] If $u_1, u_2 \in S_{\mu}$, $e(u_1) = e(u_2)$, $\max\{in(|u_1|) \} = \max\{in(|u_2|)\}$ and $\max\{in(u_1) \} < \max\{ in(u_2)\}$, then $u_1>u_2$.

\vspace{.1in}

\item[(e)] If $u_1, u_2 \in S_{\mu}$, $e(u_1) = e(u_2)$, $\max\{in(|u_1|) \} = \max\{in(|u_2|)\}$ and $\max\{in(u_1) \} = \max\{ in(u_2)\}=a_1$, then we have the following:
\smallbreak
\begin{itemize}
\item[$\alpha$.] if $\max\{in(u_1)\backslash \{a_1\}\} < \max\{in(u_2)\backslash \{a_1\}\}$, then $u_1>u_2$.
\smallbreak
\item[$\beta$.] If $\max\{in(u_1)\backslash \{a_1\}\}=\max\{in(u_2)\backslash \{a_1\}\}=a_2$, we compare $\max\{in(u_1)\backslash \{a_1, a_2\}\}$ to $\max\{in(u_2)\backslash \{a_1, a_2\}\}$ and continue if necessary until:
$$\max\{in(u_1)\backslash \{a_1, \ldots, a_c\}\} < \max\{in(u_2)\backslash \{a_1, \ldots, a_c\}\}.\quad {\rm Then}\ u_1 > u_2.$$
\end{itemize}
\end{itemize}
\end{defn}

\begin{ex}\rm
From Definition~\ref{ord2} we have that:
\[
\begin{array}{lllllll}
i. & s_1 \cdot s_3^2 & < & s_2^2\cdot s_3\cdot s_4, && & {\rm since}\ s_1 \cdot s_3^2\in S_7\ {\rm and}\ s_2^2\cdot s_3\cdot s_4\in S_{11}\\
&&&&&&\\
ii. & s_2^2           & < & s_1^2\cdot s_2 & < & s_1^4, & {\rm since}\ e(s_2^2)=2 < 3 = e(s_1^2\cdot s_2) < 4 = e(s_1^4)   \\
&&&&&&\\
iii. & s_2^2   & < & s_1\cdot s_3, &&& {\rm since}\ \max\{in(|s_2^2|) \} =2 < 3 = \max\{in(|s_1\cdot s_3|)\}\\
&&&&&&\\
iv. & s_{-5}s_1s_4s_5 & < & s_{-5}s_3^2s_4, &&& {\rm since}\ \max\{s_{-5}s_1s_4s_5\} = 5 > 4 = \max\{s_{-5}s_3^2s_4 \}\\
&&&&&&\\
v. & s_{-5}s_1s_4s_5   & < & s_{-5}s_2s_3s_5, &&& {\rm since}\ \max\{in(s_{-5}s_1s_4s_5)\backslash \{s_5\} \} =4\ {\rm and}\\
    &                  &     &                 &&&  \max\{in(s_{-5}s_2s_3s_5)\backslash \{s_5\}\}= 3 \\
\end{array}
\]
\end{ex}









\begin{prop}
The set $S$ equipped with the ordering given in Definition~\ref{ord2} is a totally ordered set.
\end{prop}

\begin{proof}
In order to show that the set $S$ is a totally ordered set when equipped with the ordering relation defined in Definition~\ref{ord2}, we need to prove that the ordering relation is antisymmetric, transitive and total. We show that the ordering relation is transitive. The fact that the ordering relation is antisymmetric follows similarly. Totality follows from the fact that all different cases have been considered. Let

\[
u_1=s_{-k_m, k_n}^{\epsilon_{-k_{m}, k_n}},\qquad u_2=s_{-l_a, l_b}^{\epsilon_{-l_a, l_b}},\qquad u_3=s_{-p_c, p_d}^{\epsilon_{-p_c, p_d}},
\]

\noindent such that $u_1\ <\ u_2$ and $u_2\ <\ u_3$. We will prove that $u_1\ <\ u_3$. Indeed:

\smallbreak

 Since $u_1\ <\ u_2$, one of the following holds:

\smallbreak

\begin{itemize}
\item[i.] If $u_1\in S_{\mu}$ and $u_2\in S_{v}$ such that $\mu < v$, then, since $u_2<u_3$ we have that if $u_3\in S_{\phi}$, then $v\leq \phi$ and thus, $u_1<u_3$.

\bigbreak

\item[ii.] If $u_1, u_2, u_3\in S_{\mu}$ and $e(u_1)<e(u_2)$, then, since $u_2<u_3$ we have that $e(u_2) \leq e(u_3)$ and thus, $u_1<u_3$.

\bigbreak

\item[iii.] If $u_1, u_2, u_3\in S_{\mu}$, $e(u_1)=e(u_2)=e(u_3)$ and $\max\{in(|u_1|) \} < \max\{in(|u_2|)\}$, then, since $u_2<u_3$ we have that $\max\{in(|u_2|) \} \leq \max\{in(|u_3|)\}$ and thus, $u_1<u_3$.

\bigbreak

\item[iv.] If $u_1, u_2, u_3\in S_{\mu}$, $e(u_1)=e(u_2)=e(u_3)$, $\max\{in(|u_1|) \} = \max\{in(|u_2|)\} = \max\{in(|u_3|)\}$ and $\max\{in(u_1) \} > \max\{in(u_2)\}$, then, since $u_2<u_3$ we have that $\max\{in(u_2) \} > \max\{in(u_3)\}$ and thus, $u_1<u_3$.

\bigbreak

\item[v.] If $u_1, u_2, u_3\in S_{\mu}$, $e(u_1)=e(u_2)=e(u_3)$, $\max\{in(|u_1|) \} = \max\{in(|u_2|)\} = \max\{in(|u_3|)\}$, $\max\{in(u_1) \} = \max\{in(u_2)\} = \max\{in(u_3)\} = j$ and $\max\{in(u_1)\backslash \{j\} \} > \max\{in(u_2)\backslash \{j\}\}$, then, since $u_2<u_3$ we have that $\max\{in(u_2)\backslash \{j\} \} \geq \max\{in(u_3)\backslash \{j\}\}$ and thus, $u_1 < u_3$.

\bigbreak

\item[vi.] If $u_1, u_2, u_3\in S_{\mu}$, $e(u_1)=e(u_2)=e(u_3)$, $\max\{in(|u_1|) \} = \max\{in(|u_2|)\} = \max\{in(|u_3|)\}$, $\max\{in(u_1) \} = \max\{in(u_2)\} = \max\{in(u_3)\} = j_1$, $\max\{in(u_1)\backslash \{j_1\} = \max\{in(u_2)\backslash \{j_1\}$, $\ldots$, $\max\{in(u_1)\backslash \{j_1, \ldots, j_{\beta-1}\} \} = \max\{in(u_2)\backslash \{j_1, \ldots, j_{\beta-1}\} \}$ and \\
$\max\{in(u_1)\backslash \{j_1, \ldots, j_{\beta}\} \} > \max\{in(u_2)\backslash \{j_1, \ldots, j_{\beta}\} \}$, we have that since $u_2<u_3$, there exist an index $c$ such that $\max\{in(u_2)\backslash \{j_1, j_2\}\} = \max\{in(u_3)\backslash \{j_1, j_2\}\}$, $\ldots$ , \\
$\max\{in(u_2)\backslash \{j_1, \ldots, j_{c-1}\} \} = \max\{in(u_3)\backslash \{j_1, \ldots, j_{c-1}\} \}$ and \\
$\max\{in(u_2)\backslash \{j_1, \ldots, j_{c}\} \} > \max\{in(u_3)\backslash \{j_1, \ldots, j_{c}\} \}$. We now consider the following cases:
\smallbreak
\begin{itemize}
\item[a.] If $\beta \leq c$, we have that\\
$\max\{in(u_1)\backslash \{j_1, \ldots, j_{\beta}\} \} > \max\{in(u_2)\backslash \{j_1, \ldots, j_{\beta}\} \} =  \max\{in(u_3)\backslash \{j_1, \ldots, j_{\beta}\}$ and thus, $u_1 < u_3$.
\smallbreak
\item[b.] If $b > c$, we have that\\
$\max\{in(u_1)\backslash \{j_1, \ldots, j_{c}\} \} = \max\{in(u_2)\backslash \{j_1, \ldots, j_{c}\} \} >  \max\{in(u_3)\backslash \{j_1, \ldots, j_{c}\}$ and thus, $u_1 < u_3$.
\end{itemize}
\end{itemize}
\noindent Thus, the ordering relation defined in Definition~\ref{ord2} is transitive.
\end{proof}

\begin{prop}
The sets $S_{(k)}$ equipped with the ordering relation defined in Definition~\ref{ord2} are totally ordered and well-ordered sets for all $k$.
\end{prop}

\begin{proof}
Since $S_{(k)} \subset S, \forall\ k$, $S_{(k)}$ inherits the property of being a totally ordered set from $S$. Moreover, according to Definition~\ref{ord2}, $s_k$ is the minimum element of $S_{(k)}$ and so $S_{(k)}$ is a well ordered set.
\end{proof}

\begin{prop}\label{resp}
The ordering relation defined in Definition~\ref{ord2} on the sets $S$ and $S_{(k)}$, respects the ordering relation defined in Definition~\ref{order} on the sets $\Lambda^{\prime}$ and $\Lambda^{\prime}_{(k)}$.
\end{prop}

\begin{proof}
Let $\tau_1^{\prime}, \tau_2^{\prime} \in \Lambda^{\prime}$ such that
\[
\tau_1^{\prime}= t^{k_0}{t_1^{\prime}}^{k_1} \ldots {t_m^{\prime}}^{k_m} < t^{l_0}{t_1^{\prime}}^{l_1} \ldots {t_n^{\prime}}^{l_n}=\tau_2^{\prime}.
\]
We show that $u_1=tr(\tau_1^{\prime})=s_{k_0}s_{k_1} \ldots s_{k_m} < u_2=tr(\tau_2^{\prime})=s_{l_0}s_{l_1} \ldots s_{l_n}$ in the ordering of Definition~\ref{ord2}.

Since $\tau_1^{\prime} < \tau_2^{\prime}$, by Definition~\ref{order} one of the following holds:
\begin{itemize}
\item[i.] $\sum_{i=0}^{m}k_i =\mu < \nu = \sum_{i=0}^{n}l_i$. Then $u_1\in S_\mu$ and $u_2\in S_\nu$, so $u_1<u_2$ by Definition~\ref{ord2}(a).
\smallskip

\item[ii.] $\mu = \nu$, but $ind(\tau_1^{\prime})=m< n=ind(\tau_2^{\prime})$. Then $e(u_1)=m < n=e(u_2)$, so $u_1<u_2$ by Definition~\ref{ord2}(b).
\smallskip

\item[iii.] $\mu=\nu$, $m=n$, and the sequences of indices agree up to some position $i$, where $|k_i|<|l_i|$. Since the exponents are ordered, the lexicographical comparison of the index sets in $u_1,u_2$ yields $u_1<u_2$ by Definition~\ref{ord2}(c)–(e).
\end{itemize}

Thus in all cases the ordering of $u_1,u_2$ respects that of $\tau_1^{\prime},\tau_2^{\prime}$.
\end{proof}


\begin{remark}\rm
It is worth mentioning that the ordering relation defined on the sets $S_{(k)}$ does not necessarily respect the ordering relation defined on the set $({\Lambda}^{aug})^{\prime}$. For example, $t^2t^{\prime}_1 > t{t^{\prime}}^2_1$, but $tr(t^2t^{\prime}_1) = tr(t{t^{\prime}}^2_1) =s_1s_2$.
\end{remark}

\begin{remark}\rm
Let $\tau\in \Lambda$ and $tr(\tau)=\underset{i}{\Sigma}a_i\cdot S_i$, where $a_i\in \mathbb{C}[q^{\pm 1}, z^{\pm 1}]$ and $S_i$ are monomials of $s_i$'s for all $i$. Let $\max\{tr(\tau)\}$ denote the maximum monomial in $s_i$'s coming from the computation of the trace of $\tau$. Then, according to Proposition~\ref{resp} and the change of basis matrix, if $\tau_1<\tau_2 \in \Lambda$, then $\max\{tr(\tau_1)\}>\max\{tr(\tau_2)\}$, by ignoring the coefficients in $\mathbb{C}[q^{\pm 1}, z^{\pm 1}]$.
\end{remark}

We now fix $k\in \mathbb{Z}$ and consider the subset of $\Lambda^{aug}$ of level $k$, $\Lambda_{(k)}^{aug}$. We perform bbm's on all elements in $\Lambda_{(k)}^{aug}$ on their first moving strand and we obtain the equations $X_{\widehat{\tau}}\ =\ X_{\widehat{bbm_1(\tau)}}, \quad {\rm for\ all}\ \tau\in \Lambda_{(k)}^{aug}$. Denote this subsystem of equations by $[\Lambda_{(k)}^{aug}]$ and by Theorem~\ref{comb}, $S_{(k)}$ denotes the set of all unknowns in $[\Lambda_{(k)}^{aug}]$. Set now $\min\{S_{(k)} \}$ to be the minimum element in $S_{(k)}$. Then, from Definition~\ref{ord2} we have that:

\begin{equation}
\min\{S_{(k)} \}\ =\  s_k
\end{equation}

\noindent and moreover, if $k\in \mathbb{N}$, then $\max\{S_{(k)} \}=s_1^k$.


\bigbreak

We will prove that in $\mathcal{S}(S^1 \times S^2)$, all elements in $S_{(k)}$ can be written as sums of elements in $S_{(k)}$ of lower order, and thus:

$$
s_{i, j}^{k_{i, j}} \ =\ a \cdot s_k,\ {\rm for\ all}\ s_{i, j}^{k_{i, j}} \in S_{(k)}\ {\rm where}\ a\in \mathbb{C}[q^{\pm 1}, z^{\pm 1}].
$$

\subsection{Dependence of the elements in $S_{(k)}$}

Let $\tau\in \Lambda$ and $\tau \sim \tau^{\prime}$, the homologous word in $\Lambda^{\prime}$. Consider the inverse of the change of basis matrix, converting elements in $\Lambda$ to sum of elements in $\Lambda^{\prime}$. Then:

\[
\tau \ = \ A\cdot \tau^{\prime} \ + \ \sum_i A_i \cdot \tau_i^{\prime}\\
\] 

\noindent where $\tau_i^{\prime}>\tau, \forall i$ and $A_i\in \mathbb{C}[q^{\pm 1}, z^{\pm 1}]$ (see also Remark~3). So, when converting elements in $\Lambda$ to sums of elements in $\Lambda^{\prime}$, we obtain the homologous word together with greater order terms. Equivalently,
we have the following:

\begin{prop}
Let $\tau_{0,m}^{k_{0,m}} \in \Lambda$. Then, the following relation holds:

$$tr(\tau_{0,m}^{k_{0,m}})=q^{\sum_{i=1}^{m}{k_i\cdot i}}s_{k_{0}, k_{m}} + \sum f(q,z) s_{\lambda_{0}, \lambda_{m}},$$

\noindent where $s_{\lambda_{0}, \lambda_{m}} > s_{k_{0}, k_{m}}$ for all $s_{\lambda_{0}, \lambda_{m}}$.
\end{prop}

\begin{proof}
It follows directly from Remark~\ref{rem2}, Definition~\ref{ord2} and the discussion above.
\end{proof}

Let now $\tau \in \Lambda_k^{aug}$. We apply a bbm on the first moving strand of $\tau$ and we have that $\tau \overset{bbm}{\longrightarrow} \tau_+\cdot \sigma_1^{\pm 1}$, where $\tau_+$ is the word $\tau$ with all indices shifted by one. As shown in \cite{DL2}, applying conjugation and stabilization moves on $\tau_+\cdot \sigma_1^{\pm 1}$, we can manage the {\it gap} in the first index of this element and obtain a sum of elements in $\Lambda^{aug}$. We recall the following result from \cite{DL2}:

\begin{lemma}[Lemma~13 \cite{DL2}]\label{oldlem}
For $k\in \mathbb{N}$, $\epsilon =1$ or $\epsilon =-1$ and $a\in {\rm H}_{a, n}(q)$, the following relations hold:
\[
t_i^{\epsilon k}\cdot a\ \widehat{=}\ q^{\epsilon (k-1)}\cdot t_{i-1}^{\epsilon k}\cdot \sigma_i^{\epsilon}\cdot a\cdot \sigma_i^{\epsilon} \ +\ \underset{u=1}{\overset{k-1}{\sum}} q^{\epsilon (u-1)} (q^{\epsilon}-1)\cdot t_{i-1}^{\epsilon u}\cdot t_i^{\epsilon (k-u)}\cdot a\cdot \sigma_i^{\epsilon}.
\]
\end{lemma}

Using now Lemma~\ref{oldlem} we have the following:

\begin{lemma}\label{newgaps}
For $k\in \mathbb{N}$, $\epsilon =1$ or $\epsilon =-1$ and $a\in {\rm H}_{a, n}(q)$, the following relations hold:
\[
\begin{array}{lll}
\tau_{1, m+1}^{k_0, m}\cdot \sigma_1 & \widehat{=} & q^{\underset{i=0}{\overset{m}{\sum}}k_i-m}\cdot \tau_{0, m}^{k_{0, m}}\cdot \sigma_{m+1} \ldots \sigma_2 \sigma_1^2 \sigma_2 \ldots \sigma_{m+1}\ +\\
&&\\
 & + & \underset{i=0}{\overset{m+1}{\sum}}\ \  \underset{u=1}{\overset{k_1-1}{\sum}} q^{u+ \underset{j=0}{\overset{i-1}{\sum}}k_j -(i+1)} (q-1) \tau_{0, i-1}^{k_{0, i-1}}\cdot t_i^u \cdot t_{i+1}^{k_{i+1-u}}\cdot \tau_{i+2, m+1}^{k_{i+1, m}}\cdot \sigma_1^2\sigma_2\ldots \sigma_{i+1}\\
\end{array}
\]
\end{lemma}

\begin{proof}
Applying Lemma~\ref{oldlem} in $\tau_{1, m+1}^{k_0, m}$ we obtain:

$$t_1^{k_0}\cdot \tau_{2, m+1}^{k_1, m}\ \widehat{=}\ q^{(k_1-1)}\cdot t^{k_1}\cdot \sigma_1^{2}\ +\ \underset{u=1}{\overset{k_1-1}{\sum}} q^{(u-1)} (q-1)\cdot t^{u}\cdot t_2^{(k_1-u)}\cdot \tau_{2, m+1}^{k_1, m}\cdot \sigma_1$$

\noindent We apply again Lemma~\ref{oldlem} in each monomial on the right hand side of the relation that contains gaps in the indices until we reach at monomials in the $t_i$'s with consecutive indices. The result follows.
\end{proof}

\begin{remark}\label{rmk4}\rm
It is worth mentioning that when managing the gaps in the indices of the monomial $\tau_{1, m+1}^{k_0, m}\sigma_1^{\pm 1}$, which is obtained from $\tau_{0, m}^{k_0, m}$ by applying bbm on the first moving strand, we obtain $\tau_{0, m}^{k_0, m}$ followed by a `braiding tail' and elements in the $\Lambda^{aug}$ ${\rm H}_n(q)$-module, some of which are of lower order than $\tau_{0, m}^{k_0, m}$ and some of which are of higher order than $\tau_{0, m}^{k_0, m}$. Then, as shown in \cite{DL2} and explained above, using conjugation and stabilization moves these elements can be expressed as sums of elements in $\Lambda^{aug}$ whose order, with respect to $\tau_{0, m}^{k_0, m}$, \underline{varies}.
\end{remark}

In order to prove the main result of this paper we first present some lemmas that were the motivation for the computation of $\mathcal{S}(S^1 \times S^2)$.

\begin{lemma}\label{newlem}
For $k \in \mathbb{N}$ the following relations hold:
\[
\begin{array}{llll}
i. & tr(t^pt_1^kg_1) & = & q^k z s_{p+k}\ + \ \sum_{j=0}^{k-1}{q^j(q-1)tr(t^{p+j}t_1^{k-j})} \\
&&&\\
ii. & tr(t^pt_1^kg_1^{-1}) & = & q^{k-1} z s_{p+k}\ +\ \sum_{j=0}^{k-2}{q^j(q-1)tr(t^{p+1+j}t_1^{k-1-j})} \\
&&&\\
iii. & tr(t^pt_1^kg_1) & = & (q^2-q+1)tr(t^{p+1}t_1^{k-1}g_1)\ +\  q^k(q-1)s_ks_p\ +\\
&&&\\
&& + & \sum_{j=2}^{k}{q^{j-1}(q-1)^2tr(t^{p+k-j}t_1^{j}g_1)}\\
&&&\\
iv. & tr(t^pt_1^k) & = & \sum_{j=0}^{k}{f_j(q,z)s_{p+k-j}s_{j}}
\end{array}
\]
\end{lemma}

\begin{proof}
We only prove relations (i) and (iv). Relations (ii) and (iii) follow similarly.

\smallbreak

\begin{itemize}
\item[i.] We prove relations (i) by induction on $k$. For $k=1$ we have that:

$$tr(t^pt_1g_1)=(q-1)tr(t^pt_1)+qtr(t^pg_1t)=(q-1)tr(t^pt_1)+qzs_{p+1},$$ 

\noindent and so the relation holds for $k=1$. Assume that the relation holds for $k-1$. Then for $k$ we have:

\[
\begin{array}{lllc}
tr(t^pt_1^kg_1) & = & (q-1)tr(t^pt_1^k)\ +\ qtr(t^{p+1}t_1^{k-1}g_1) & \overset{ind. step}{=}\\
&&&\\
& = & (q-1)tr(t^pt_1^k)\ +\ q\sum_{j=0}^{k-2}{q^j(q-1)tr(t^{p+1+j}t_1^{k-1-j})}\ +\ q^k z s_{p+k} & \overset{j=u-1}{=}\\
&&&\\
& = & q^k z s_{p+k}\ +\  q^0(q-1)tr(t^pt_1^k)\ +\ \sum_{u=1}^{k-1}{q^u(q-1)tr(t^{p+u}t_1^{k-u})} & =\\
&&&\\
& = & q^k z s_{p+k}\ + \ \sum_{u=0}^{k-1}{q^u(q-1)tr(t^{p+u}t_1^{k-u})}. &\\
\end{array}
\]
\\

\item[iv.] We have that: $t_1^k  =  q^{k}{t_1^{\prime}}^k\ -\ \sum_{j=1}^{k}{q^{j}(q^{-1}-1)t^{j-1}t_1^{k+1-j}g_1^{-1}}$ and so

$$tr(t^pt_1^k)\ =\ q^ks_ks_p\ -\ \sum_{j=1}^{k}{q^j(q^{-1}-1)tr(t^{p+j-1}t_1^{k+1-j}g_1^{-1})}.$$

\noindent We also have that:

$$tr(t^{p+j-1}t_1^{k+1-j}g_1^{-1})\ = \ q^{k-j} z s_{p+k}\ +\ \sum_{i=0}^{k-1-j}{q^i(q-1)tr(t^{p+i+j}t_1^{k-i-j})}\ {\rm and\ thus}$$

\[
\begin{array}{lcll}
tr(t^pt_1^k) & = & q^ks_ks_p\ - kq^k(q^{-1}-1)zs_{p+k}\ -&\\
&&&\\
& - &  \sum_{j=1}^{k}{\left(\sum_{i=0}^{k+1-j}{q^{j+i}(q-1)(q^{-1}-1)tr(t^{p+j+i}t_1^{k-j-i})}\right)} & (a).
\end{array}
\]

\noindent We now prove relations (iv) by strong induction on $k$. For $k=1$ we have $tr(t^pt_1)=qs_1s_p+(q-1)zs_{p+1}$, which holds.

\smallbreak

\noindent Assume that it holds for $k=1, \ldots, m-1$. Then for $k=m$ we have that:

\[
\begin{array}{lcl}
tr(t^pt_1^m) & \overset{(a)}{=} & q^ms_ms_p-mq^m(q^{-1}-1)zs_{m+p}\ -\\
&&\\
& - & \sum_{j=1}^{m}{\left(\sum_{i=0}^{m+1-j}{q^{j+i}(q-1)(q^{-1}-1)\underline{tr(t^{p+j+i}t_1^{m-j-i})}}\right)}\ \overset{ind. step}{=}\\
&&\\
& = & q^ms_ms_p-mq^m(q^{-1}-1)zs_{m+p}\ -\\
&&\\
& - & \sum_{j=1}^{m}{\sum_{i=0}^{m+1-j}{q^{j+i}(q-1)(q^{-1}-1)\left(\sum_{r=0}^{m-i-j}{f_r(q,z)s_{p+m-r}s_{r}} \right)}}\ =\\
&&\\
& = & \sum_{\mu=0}^{m}{f_{\mu}(q,z)s_{p+k-\mu}s_{\mu}}
\end{array}
\]

\end{itemize}
\end{proof}

Using now Lemma~\ref{newlem} we prove that elements in the $s_i$'s are dependent in $\mathcal{S}(S^1\times S^2)$. More precisely we have the following:

\begin{prop}
For $k\in \mathbb{N}$ the following holds in $\mathcal{S}(S^1 \times S^2)$:
\begin{itemize}
\item[i.] $s_k\ =\ \sum_{i=1}^{k-1} A_i\cdot s_{i}s_{k-i}$,
\smallbreak
\item[ii.] $\sum_{i, j\in \mathbb{N},\ i+j=k} s_is_j\ =\ \sum_{i, j, h\in \mathbb{N},\ i+j+h=k} B_{i,j,h}\cdot s_is_js_h$
\end{itemize}

\noindent where $A_i, B_{i,j,h} \in \mathbb{C}[q^{\pm 1}, z^{\pm 1}]$.
\end{prop}

\begin{proof}
We only prove relations (i). Relations (ii) follow similarly.

\bigbreak

Consider the element $t^k$ and perform a positive {\it bbm} to obtain $t_1^k\cdot \sigma_1$. Then, we obtain the equation $X_{\widehat{t^k}}=X_{\widehat{t_1^k\cdot \sigma_1}}$. Since the trace is linear, applying Lemma~\ref{oldlem} on $tr(t_1^k\cdot \sigma_1)$ we obtain:

\[
tr(t_1^k\cdot \sigma_1)\ =\ q^{k-1}\cdot tr(t^k\cdot \sigma_1^3)\ +\ \sum_{u=1}^{k-1}q^{u-1}(q-1)\cdot tr(t^{u}t_1^{k-u}\cdot \sigma_1^2)
\]

\noindent and by simple calculations we have that $tr(t^k\cdot \sigma_1^3)=s_k\cdot [(q^2-q+1)\cdot z+q(q-1)]$. Applying also Lemma~\ref{newlem} on all $\sum_{u=1}^{k-1} tr(t^{u}t_1^{k-u}\cdot \sigma_1^2)$ we obtain:

\[
s_k\ =\ \sum_{i=1}^{k-1} A_i\cdot s_{i}s_{k-i},\ \ {\rm for}\ A_i\in \mathbb{C}[q^{\pm 1}, z^{\pm 1}].
\]
 \noindent and this concludes the proof.
\noindent 
\end{proof}

\subsection{The main result}

We are now in position to prove the main result of this paper.

\begin{thm}\label{depthm}
The free part of the HOMFLYPT skein module of $S^1\times S^2$, $\mathcal{S}(S^1\times S^2)$, is generated by the empty knot.
\end{thm}

\begin{proof}

Consider an element $\tau \in \Lambda^{aug}$ and perform a bbm on the first moving strand. We obtain $\tau_+\cdot \sigma_1^{\pm 1}$, and by applying Lemma~\ref{newgaps} and Remark~\ref{rmk4} we have that $\tau_+\sigma_1 \ \widehat{\cong}\ \sum_i A_i\cdot \tau_i$, where $A_i\in \mathbb{C}[q^{\pm 1}, z^{\pm 1}]$ and $\tau_i\in \Lambda^{aug}_k, \forall i$. We now convert the elements $\tau_i$ to sum of elements in $\Lambda^{\prime}$ using the inverse of the infinite matrix and we obtain:

\[
\tau_+\cdot \sigma_1^{\pm 1}\ \widehat{\cong}\ A\cdot \tau^{\prime}\ +\ \sum_j A_j\cdot \tau_j^{\prime},\ \ {\rm where}\ \tau^{\prime} \sim \tau\ {\rm and}\ \tau_i^{\prime}\ {\rm are\ of\ various\ order\ with\ respect\ to}\ \tau^{\prime}.
\] 

\noindent Set now $T\ =\  \sum_j A_j\cdot \tau_j^{\prime}$. We consider the following cases:

\bigbreak

\begin{itemize}
\item[Case\ A.] If $\tau_j^{\prime}\ >\ \tau^{\prime}$ for all $j$, then $X_{\widehat{\tau^{\prime}}}\ =\ X_{\widehat{T}}\ \Rightarrow\ tr(\tau^{\prime})\ =\ \underset{i}{\Sigma}b_i \cdot tr(\tau_i^{\prime})\ \Rightarrow\ s\ =\ \underset{i}{\Sigma}b_i \cdot S_i$, where $b_i\in \mathbb{C}[q^{\pm 1}, z^{\pm 1}]$. Thus, $s\in S_k$ is written as a sum of higher order terms in $S_k$.

\bigbreak

\item[Case\ B.] If there exist a $j$ such that $\tau_j^{\prime}\ \leq \ \tau^{\prime}$, then set $s\in S_k$ the minimum monomial in $s_i$'s that appear in the computation of $tr(\tau^{\prime})$ and $tr(\tau_j^{\prime})$ and express this element $s$ as a sum of higher order terms in $S_k$, as in Case A.
\end{itemize}

\noindent We conclude that in both cases, the monomials in $s_i$'s appearing in $S_k$ can be expressed as sums of other elements in $S_k$ of \emph{higher order}. By a general position argument, we may reverse this process: every element in $S_k$ can equivalently be written as a sum of elements of \emph{lower order} in $S_k$. Then, by induction on the ordering in $S_k$, we deduce that any $s \in S_k$ satisfies
\[
s = a \cdot \min\{S_k\} = a \cdot s_k,
\]
for some scalar $a \in \mathbb{C}[q^{\pm1}, z^{\pm1}]$. Thus, the value of every element of $S_k$ in $\mathcal{S}(S^1 \times S^2)$ depends only on the value of the minimal element $s_k$.

\smallbreak

Consider now the element $t^k$ and apply the procedure described before in order to obtain the equation for the infinite system:
$s_k\ =\ \underset{i}{\Sigma} c_i\cdot s_i$, where $s_i \geq s_k$ for all $i$. This is illustrated in the following diagram:

\[
\begin{array}{cccl}
s_k & \overset{tr}{\longleftarrow} \ {\Lambda^{\prime}}^{aug} \ni t^k & \overset{bbm}{\rightarrow}  & t_1^k\cdot \sigma_1^{\pm 1} \\
\uparrow  &                              &                                 & {\downarrow}\ {\rm conj.} \ \&\ {\rm stab.\ moves} \\
|	&                                                &  &  a\cdot t^k + \underset{i}{\Sigma} a_i\cdot t^{u_i}t_1^{k-u_i}\\
\downarrow	&                              &                                                             &\downarrow {\rm matrix} \\
\underset{i}{\Sigma} c_i\cdot s_i	&     \overset{tr}{\longleftarrow}             &                              & a\cdot t^k + \underset{i}{\Sigma} a_i\cdot t^{v_i}{t_1^{\prime}}^{k-v_i}  \\
\end{array}
\]

From Theorem~\ref{gp}, all these monomials can be written as $a\cdot s_k$, and thus in $\mathcal{S}(S^1\times S^2)$ we have that $b\cdot s_k\ =\ 0$, for some $b\in \mathbb{C}[q^{\pm 1}, z^{\pm 1}]$. Thus, $s_k$ is also a torsion element in $\mathcal{S}(S^1\times S^2)$. Moreover, the equations for the infinite system of equations obtained by performing bbm's on the unknot are:

\[
1\overset{bbm}{\rightarrow}\sigma_1^{\pm 1} \Leftrightarrow e=e
\]

\noindent and thus, the unknot does not produce torsion in $\mathcal{S}(S^1\times S^2)$. 

\end{proof}

\begin{ex}\rm
In this example we present some calculations corresponding to $S_k$ for $k=0, 1$ and $2$.

\smallbreak

- For the element $t \in \Lambda_1$ we have:

\[
t\overset{bbm}{\rightarrow}t_1\sigma_1^{\pm 1} \Leftrightarrow s_{1}=\frac{\lambda}{z}\cdot \left[q\cdot (q-1)+(q^2-q+1)\cdot z \right]\cdot s_{1}
\]

\bigbreak

- For the element in $t^2 \in \Lambda_2$ we have:

\[
\begin{matrix}
	 
	& & s_2 = \frac{\lambda\cdot q\cdot (q-1)^2}{z\cdot \left[1-\lambda\cdot (q^2-q+1) \right]}\cdot s_1^2 \\
	
t^2 \overset{bbm}\rightarrow t_1^{2}\sigma_1^{\pm 1} & \Leftrightarrow & {\rm and}\\
		
	& &	s_2 = \frac{\lambda^2\cdot q\cdot (q-1)^3+q^2\cdot (q-1)}{z-\lambda^2\cdot \left[(q-1)^4\cdot z+3\cdot q\cdot (q-1)^2\cdot z+ q^2\cdot (q-1)+q^2\cdot z \right]}\cdot s_1^2\\
	\end{matrix}
\]

\bigbreak

- For the element $t^{-1}t_1 \in \Lambda_0$ we have:

\[
\begin{matrix}
	 
	& & s_{-1}s_1 = \frac{-z\cdot ( z+q-1 )}{q}\cdot e \\
	
t^{-1}t_1 \overset{bbm}\rightarrow t_1^{-1}t_2\sigma_1^{\pm 1} & \Leftrightarrow & {\rm and}\\
		
	& &	s_{-1}s_1 = \frac{(q^{-2}+1)\cdot z^2+(q^{-1}-1)^2\cdot (q-1)\cdot z}{z-q\cdot (q^{-1}-1)^2}\cdot e\\
	\end{matrix}
\]
\end{ex}

\section{Conclusions}

In this paper we computed the HOMFLYPT skein module of \(S^1 \times S^2\) via the braid-theoretic approach. We developed a systematic method based on the infinite system of equations derived from braid band moves, showing that all elements of the skein module can ultimately be expressed in terms of the empty knot. We proved that the free part of the HOMFLYPT skein module of \(S^1 \times S^2\) is generated by the empty knot and identified the presence of torsion, arising naturally from the relations imposed by the braid band moves.

It is worth emphasizing that \(S^1 \times S^2\) serves as the first known example of a 3-manifold where torsion in its HOMFLYPT skein module is detected via the braid-theoretic method. The methods developed here lay the groundwork for extending similar techniques to other 3-manifolds, notably lens spaces \(L(p, q)\) and beyond.

\end{document}